\documentclass[a4paper,12pt]{amsart}
\usepackage{hyperref}

\usepackage{microtype}
\DeclareRobustCommand{\SkipTocEntry}[5]{}

\usepackage{amssymb,amsmath,amsfonts,amsthm,bm,latexsym,amscd,verbatim,url,nicefrac,stmaryrd,enumerate,colonequals,dsfont,appendix}
\usepackage[all]{xy}

\usepackage{times}
\usepackage{graphicx}
\usepackage{fullpage}

{\bf}{\it}
\newtheorem*{thm*}{Theorem}
\newtheorem{thm}{Theorem}[section]{\bf}{\it}
\newtheorem{prop}[thm]{Proposition}
\newtheorem{lemma}[thm]{Lemma}
\newtheorem{cor}[thm]{Corollary}

\theoremstyle{definition}
\newtheorem{dfn}[thm]{Definition}
\theoremstyle{remark}
\newtheorem{rmk}[thm]{Remark}
\theoremstyle{remark}
\newtheorem{exm}[thm]{Example}
\newtheorem{assu}[thm]{Assumption}

\newcommand{\B}{\mathbb{B}}
\newcommand{\C}{\mathbb{C}}

\newcommand{\HH}{\mathbb{H}}
\newcommand{\LL}{\mathbb{L}}
\newcommand{\N}{\mathbb{N}}
\newcommand{\Q}{\mathbb{Q}}
\newcommand{\RR}{\mathbb{R}}
\newcommand{\R}{\mathbb{R}}

\newcommand{\Z}{\mathbb{Z}}

\newcommand{\cat}{\mathbf{C}}
\newcommand{\catD}{\mathbf{D}}

\newcommand{\catT}{\mathbf{T}}

\newcommand{\ra}{\rightarrow}

\newcommand{\mcF}{\mathcal{F}}
\newcommand{\mcG}{\mathcal{G}}

\newcommand{\mcO}{\mathcal{O}}
\newcommand{\mcU}{\mathcal{U}}
\newcommand{\mcV}{\mathcal{V}}
\newcommand{\mcX}{\mathcal{X}}

\newcommand{\MW}{\mathcal{MW}}

\newcommand{\del}{\partial}
\newcommand{\adj}[4]{#1\negmedspace: #2\rightleftarrows #3:\negmedspace #4}

\DeclareMathOperator{\an}{an}
\DeclareMathOperator{\Aff}{Aff}
\DeclareMathOperator{\AffSm}{AffSm}

\DeclareMathOperator{\Berk}{Berk}
\DeclareMathOperator{\car}{char}

\DeclareMathOperator{\cp}{cp}
\DeclareMathOperator{\ct}{ct}

\DeclareMathOperator{\dR}{dR}
\DeclareMathOperator{\eff}{eff}
\DeclareMathOperator{\et}{\acute{e}t}

\DeclareMathOperator{\Ev}{Ev}

\DeclareMathOperator{\FormDA}{\bf{FormDA}}

\DeclareMathOperator{\Frac}{Frac}

\DeclareMathOperator{\gm}{gm}

\DeclareMathOperator{\Hom}{Hom}

\DeclareMathOperator{\id}{id}
\DeclareMathOperator{\im}{Im}

\DeclareMathOperator{\Int}{Int}

\DeclareMathOperator{\odR}{odR}

\DeclareMathOperator{\rig}{rig}
\DeclareMathOperator{\Rig}{Rig}

\DeclareMathOperator{\RigSm}{RigSm}
\DeclareMathOperator{\RHom}{Hom_\bullet}

\DeclareMathOperator{\Sing}{Sing}
\DeclareMathOperator{\Sm}{Sm}

\DeclareMathOperator{\Spa}{Spa}
\DeclareMathOperator{\Spec}{Spec}

\DeclareMathOperator{\Sus}{Sus}
\DeclareMathOperator{\Tot}{Tot}

\DeclareMathOperator{\uhom}{\underline{Hom}}

\DeclareMathOperator{\Ch}{\bf{Ch}}

\DeclareMathOperator{\DA}{\bf{DA}}

\DeclareMathOperator{\FSH}{\bf{FSH}}
\DeclareMathOperator{\FormSm}{{FormSm}}
\DeclareMathOperator{\LRS}{\bf{LRS}}

\DeclareMathOperator{\Mod}{\bf{-Mod}}

\DeclareMathOperator{\RigDA}{\bf{RigDA}}

\DeclareMathOperator{\Psh}{\bf{Psh}}

\DeclareMathOperator{\Sh}{\bf{Sh}}
\DeclareMathOperator{\SH}{\bf{SH}}
\DeclareMathOperator{\SSpect}{\bf{Spt}}

\begin{document}
	\title{The Monsky-Washnitzer and the overconvergent realizations}
	\author{Alberto Vezzani} 
	\address{%
		LAGA, Universit\'e Paris 13\\
		99 avenue Jean-Baptiste Cl\'ement \\
		93430 Villetaneuse, France}
	\email{vezzani@math.univ-paris13.fr}
	\begin{abstract}   
		We construct the dagger realization functor for analytic motives over non-archimedean fields of mixed characteristic, as well as the Monsky-Washnitzer realization functor for algebraic motives over a discrete field of positive characteristic. In particular,  the motivic language on the classic \'etale site    provides a new direct definition of the overconvergent de Rham cohomology and rigid cohomology and shows that their finite dimensionality follows formally from  one of Betti cohomology for smooth projective complex varieties. 
	\end{abstract}

	\maketitle

	%
	%
	%
	%

	\tableofcontents

	\section{Introduction}
	
	The problem of the definition of a well-behaved cohomology theory ``\`a la de Rham'' for  analytic varieties over a non-archimedean field lies on the pathological properties of the de Rham complex in this context. Even though it behaves as expected when applied to proper varieties, or analytifications of algebraic varieties, its (hyper)cohomology computed on very basic affinoid smooth rigid varieties  (such as the closed disc $\B^1$) is oddly infinite dimensional. This is related to the impossibility to integrate general holomorphic rigid forms while preserving their radius of convergence. 
	
	This classical problem has been studied and positively resolved by  different authors (see \cite{fvdp}, \cite{gk-dR}, \cite{mw-fc1}). The common strategy is to consider the (hyper)cohomology of an alteration of the de Rham complex, namely the  \emph{overconvergent} complex $\Omega^{\dagger}$ that is, the subcomplex of those forms which can be extended on a ``strict neighborhood'' of the variety. In order to  give sense to this definition one needs first to consider an affinoid variety and endow it with an overconvergent structure, which amounts to embedding it in the interior of a bigger one. Secondly, one has to prove than two different choices would induce two canonically equivalent cohomology groups, and that this definition can be extended functorially to arbitrary varieties.

	The technical tool which is behind these facts is  the rigid version of Artin's approximation theorem proved by Bosch \cite{bosch-artin} stating that any map of varieties can be approximated with a new one preserving two given overconvergent structures, combined with a homotopic argument. In this article, we follow the approach of Gro\ss e-Kl\"onne \cite{gk-over} and \cite{gk-dR} and we  give a motivic version of this procedure which can be stated as follows (see Theorem \ref{main}):
	
	\begin{thm*}
		For any ring of coefficients $\Lambda$ the canonical functor $l\colon \RigSm^\dagger\!/K\ra\RigSm/K$ from smooth rigid varieties with an overconvergent structure to rigid varieties  induces a monoidal, triangulated equivalence of the associated categories of motives:
		$$
		\LL l^*\colon \RigDA^{\dagger\eff}_{\et}(K,\Lambda)\cong\RigDA^{\eff}_{\et}(K,\Lambda)\colon \RR l_*.
		$$
	\end{thm*}
	
	We will actually prove a relative statement, where the base dagger variety is not necessarily the spectrum of the field $K$. We remark that, in particular, it is possible to define the spectrum of the overconvergent de Rham cohomology as the motive $\LL l^*\Omega^{\dagger}$. %
	
	The main technical ingredient for the proof is Proposition \ref{sollevacocu} which provides a way to approximate maps of rigid analytic varieties having an overconvergent structure, with maps that preserve such structures. It is reminiscent of (a cubical version of) Artin's approximation lemma, but completely independent of it. 
	
	We can apply the previous theorem to give a new definition of rigid cohomology, which is a good cohomological theory ``\`a la de Rham'' for algebraic varieties $X$ over a discrete field $k$ of positive characteristic. The idea, due to Monsky and Washnitzer \cite{mw-fc1} is to find (whenever possible) a smooth formal model $\mcX$ of  $X$ over the ring of integers $ K^\circ$ of a mixed-characteristic valued field $K$ with residue $k$, and then consider the overconvergent de Rham cohomology of the associated generic fiber $\mcX_\eta$. Also in this case, the major task which is solved in the  literature consists in proving that this definition does not depend on the various choices made at each step, as well as in generalizing it to arbitrary varieties. Classically, the tools which are used include the   convergent site of Ogus \cite{ogus-ct} and the crystalline site of Berthelot \cite{berth-cris}  for proper varieties, or the overconvergent site developed by Le Stum \cite{lestum-os}.

	Using the motivic language, these problems are alternatively  solved by the following remark: a theorem of Ayoub \cite[Corollary 1.4.24]{ayoub-rig} states that the special-fiber functor $(\cdot)_\sigma\colon \FormSm/ K^\circ\ra\Sm/k$ from the category  of smooth formal schemes over  $ K^\circ$ to the category of smooth varieties over the residue field $k$ induces an equivalence of motives. In particular, by letting $(\cdot)_\eta\colon \FormSm/ K^\circ\ra\RigSm/K$  be the generic-fiber functor, there is a motivic triangulated monoidal  functor:{\small
	$$\MW^*\colon\DA^{\eff}_{\et}(k,\Lambda)\xrightarrow[\sim]{\RR(\cdot)_{\sigma*}}\FormDA^{\eff}_{\et}( K^\circ,\Lambda)\xrightarrow{\LL(\cdot)^*_\eta}\RigDA^{\eff}_{\et}(K,\Lambda)\xrightarrow[\sim]{\RR l_*}\RigDA^{\dagger\eff}_{\et}(K,\Lambda).$$
}
	
	As a whole, 
	by considering the functor $\MW^*$ and the complex $\Omega^\dagger$ we therefore obtain automatically a functorial cohomology theory on algebraic varieties over $k$ satisfying \'etale descent and homotopy invariance, which coincides with the one of Monsky-Washnitzer whenever this one is defined. It is formal to show that $\MW^*$ has a right adjoint $\MW_*$ and that 
	the motive $\MW_*\Omega^\dagger$ represents the ``classic" rigid cohomology, providing an alternative to its usual definition and to the rigid spectrum considered by Deglise-Mazzari \cite{dm} and Milne-Ramachandran \cite{mr} following Besser \cite{besser}. Our construction only uses canonical, explicit functors,  the classic \'etale sites on algebraic and analytic varieties and no hypothesis on the valuation of $K$. 
	
	Another crucial fact which is proved in the  literature concerns the finite dimensionality of the cohomological theories mentioned above (the most general statements are in \cite{kedlaya-fin}), as well as their compatibility with base extensions. The classic proofs rely on several reduction procedures, involving resolutions of singularities, localizations and homotopy. They decompose  the general statement into direct, computable checks on varieties of a special kind, such as the ones which are projective and smooth. We remark that these \textit{ad hoc} constructions are encapsulated in a fundamental theorem of Ayoub \cite[Theorem 2.5.35]{ayoub-rig}. When combined with the results of \cite{vezz-DADM}, it states that the category of rigid analytic motives with rational coefficients over a base field  $\RigDA^{\eff}_{\et}(K,\Q)$ is generated, in a suitable sense, by the motives $M(X)$ associated with smooth projective \emph{algebraic} varieties $X$ over $K$. Admittedly, the proof of the theorem consists in an elaborated composition of the standard reduction procedures, enhanced with the triangulated language allowed by the motivic setting. 
	
	As shown above, the main outcome of this article is proving that the overconvergent de Rham cohomology and rigid cohomology factor over the triangulated category $\RigDA^{\eff}_{\et}(K,\Q)$. In particular, using the theorem of Ayoub, we deduce  (see Corollary \ref{findim}) their finite-dimensionality as well as their compatibility with base change  by reducing to the  motives $M(X)$ of the aforementioned form, and hence to well-known facts related to the  classic de Rham cohomology of {complex smooth projective} varieties $X(\C)$. Our proof makes no distinction between the discrete-valuation case and the general case, and is independent on the  classic proofs (see \cite{berk-int}, and partial results in \cite{berth-fin}, \cite{gk-fin}, \cite{mebkhout-fin}).

	In the Appendix, we prove that an overconvergent structure of a variety corresponds to a presentation of its (adic) compactification as an inverse limit (in a weak sense defined by Huber) of strict inclusions of rigid varieties. This connects the theory of dagger spaces of Gro\ss e-Kl\"onne \cite{gk-over} to the theory of adic spaces of Huber \cite{huber2} and strengthens the parallel between the techniques used in this paper and the ones of \cite{vezz-fw} where smooth perfectoid spaces arise as inverse limits of finite maps of rigid varieties.

	\section{Overconvergent Rigid Varieties}\label{ovar}
	
	From now on we fix a complete valued field $K$ endowed with a non-archimedean valuation of rank $1$ and residue characteristic $p>0$. We denote by $\pi$ a pseudo-uniformizer of $K$ that is, an invertible, topologically nilpotent element. We also denote by $K^\circ$ the ring of integers and by $k$ the residue field.  We consider rigid analytic varieties as adic spaces, using the language of Huber \cite{huber2}. In particular, when we consider a point  $x\in X$ for a variety $X$ we mean a point in the sense of Huber (or, equivalently, a point of the $G$-topos of $X$). We only consider rigid analytic varieties  over $K$ which  are separated and taut (that is,  the closure of a quasi-compact subset is quasi-compact, see \cite[Definition 5.1.2]{huber}). If $R$ is a Tate algebra, we sometimes denote by $\Spa R$ the associated affinoid space $\Spa(R,R^\circ)$.

	The starting point to define \emph{overconvergent}, or \emph{dagger} varieties are the so-called  dagger algebras. For the sake of completeness, we report here their definition and some basic properties, proved in \cite{gk-over} and \cite{mw-fc1}. We refer to the Appendix for a link  between these definitions and the language of adic spaces of Huber.

	\begin{dfn}[{\cite{mw-fc1}, \cite{gk-over}}]\label{defdag}
		For $c\in K$ and $m,d\in\N_{>0}$ we denote by $K\langle c^{m/d}{\tau_1,\ldots,c^{m/d}\tau_n}\rangle $ or simply by $K\langle c^{m/d}\underline{\tau}\rangle  $ the subring of $K\langle{\tau_1,\ldots,\tau_n}\rangle$ of those power series   $\sum a_\alpha \underline{\tau}^\alpha$ such that $\lim |a_\alpha| \lambda^{|\alpha|}=0$ for $\lambda=|c|^{m/d}$. 
		We denote by $K\langle\underline{\tau}\rangle^\dagger=K\langle \tau_1,\ldots,\tau_n\rangle^\dagger$ the following topological subring of $K\langle \tau_1\ldots,\tau_n\rangle$  $$\varinjlim_h K\langle\pi^{1/h}\underline{\tau}\rangle=\varinjlim_h K\langle\pi^{1/h}{\tau_1,\ldots,\pi^{1/h}\tau_n}\rangle$$ 
		that is, the ring of those power series $\sum a_\alpha \underline{\tau}^\alpha$ such that $\lim |a_\alpha| \lambda^{|\alpha|}=0$ for some $\lambda\in\R_{>1}$. 
		A \emph{dagger algebra }is a topological $K$-algebra $R $ isomorphic to a quotient  $K\langle\underline{\tau}\rangle^\dagger\!/I$ of $K\langle\underline{\tau}\rangle^\dagger$. Its completion $\hat{R}$ is the Tate algebra $K\langle\underline{\tau}\rangle/I$.  A morphism of dagger algebras $R  \ra S$ is a $K$-linear (hence continuous) morphism. The category  $\Aff^\dagger$ of \emph{affinoid dagger spaces} is  the opposite category of dagger algebras. We denote by $\Spa^\dagger R $ the object on $\Aff^\dagger$ associated with $R $.  We say that its \emph{limit} is $\Spa \hat{R}$ where $\hat{R} $ is the completion of $R $ and vice-versa we say that $\Spa^\dagger R $ is a dagger structure of $\Spa \hat{R}$. We say that $\Spa^\dagger R $ [resp. a morphism $\Spa^\dagger R'\ra\Spa^\dagger R$] has the property $\mathbf{P}$ if $\Spa \hat{R} $ [resp. the induced morphism of affinoid spaces $\Spa\hat{R}' \ra\Spa \hat{R} $] has the property $\mathbf{P}$. The category of smooth morphisms of affinoid dagger spaces $X \ra S $ to a  fixed affinoid dagger space $S=\Spa^\dagger R $  is denoted by $\AffSm^\dagger\!/S $.  Similarly, a collection of morphisms $\{\Spa^\dagger R_i \ra\Spa^\dagger R \}$ in $\AffSm/S $ is a \emph{cover} if the induced collection $\{\Spa \hat{R}_i \ra\Spa \hat{R} \}$  is a topological cover, that is, if   $\bigsqcup\Spa \hat{R}_i \ra\Spa \hat{R}$ is surjective.
	\end{dfn}
	
	\begin{rmk}
		The functor $\Spa^\dagger R\mapsto\Spa \hat{R}$ is faithful, and $\Hom(\Spa^\dagger R',\Spa^\dagger R)$ corresponds to those maps in $\Hom(\Spa \hat{R}',\Spa\hat{R})$ such that the image of $R\subset \hat{R}$ lies in $R'\subset\hat{R}'$.
	\end{rmk}
	
	\begin{exm}\label{daggerballname}
		We denote by $\B^{n\dagger}$ the smooth affinoid dagger space $\Spa^\dagger  K\langle\tau_1,\ldots,\tau_n\rangle^\dagger$.  
	\end{exm}

	\begin{prop}[{\cite[Paragraph 1.4]{gk-over}}]
		The dagger algebra $K\langle\underline{\tau}\rangle^\dagger$ is  a Noetherian factorial Jacobson ring. In particular, any dagger algebra  is isomorphic to a quotient $K\langle\tau_1,\ldots,\tau_n\rangle^\dagger\!/(f_1,\ldots,f_k)$ with $f_i\in K\langle \pi^{1/N}\underline{\tau}\rangle$ for a sufficiently big $N$.
	\end{prop}

	\begin{dfn}
		Choose a presentation of a  dagger algebra $R \cong K\langle\tau_1,\ldots,\tau_n\rangle^\dagger\!/(f_1,\ldots,f_k)$ with completion $\hat{R}$. We   denote by $\hat{R}_h$ the Tate algebra $K\langle\pi^{1/(H+h)}\underline{\tau}\rangle/(f_i)$. It is well defined for all $h\geq1$ for a sufficiently big $H$. 
		The ring $R $ is the union $\varinjlim_h \hat{R}_h$. If we denote  by $X $ the space $\Spa^\dagger R $ we also denote by $X_h$ the space $\Spa \hat{R}_h$ and by $\hat{X}$ the space $\Spa \hat{R}$. 
	\end{dfn}
	
	\begin{rmk}
		The definition of the algebras $\hat{R}_h$  above depend on the presentation of the dagger algebra $R$. Whenever we use this notation, we consider a possible choice of presentation of $R$. If we let $R^{ +}$ be $\varinjlim \hat{R}_h^\circ$ then the affinoid Huber space $\Spa(R ,R^{ +})$ is the compactification of $\hat{X}$ over $K$ and is a (weak) inverse limit of adic spaces $\varprojlim X_h$ following Huber's definition \cite[Definition 2.4.2]{huber}. We refer to the Appendix for the details (see Proposition \ref{moraffdag}).
	\end{rmk}
	
	\begin{rmk}\label{dagislim2}
		Let $U$ and $V$ be two open subvarieties of $X$. We write $U\Subset_XV$ if the closure of $U$ lies in $V$ (see \cite[Proposition 2.1.13]{ayoub-rig}). By \cite[Proposition 2.1.16]{ayoub-rig} we have that $\hat{X}\Subset_{X_1} X_{h}\Subset_{X_1}X_1$. Moreover, the sequence $\{X_h\}$ is coinitial with respect to $\Subset$ among rational subspaces of $X_1$ with this property.
	\end{rmk}
	
	We now recall some basic facts about the category of dagger spaces. In particular, we isolate in the following proposition the fundamental Artin's approximation lemma. It will not be used  under this general form, but rather in a smooth ``cubical'' fashion (see \ref{sollevacocu}).
	
	\begin{prop}[{\cite[Corollary 7.5.10]{fvdp}}]Suppose $\car K=0$ or $\car K=p>0$ and $[K\colon K^p]<\infty$. Let $X  $ and $Y  $ be two  affinoid dagger spaces with limit $\hat{X}$ and $\hat{Y}$ respectively. We fix a Banach norm $||\cdot||$ on $\mcO(\hat{X})$ and $\mcO(\hat{Y})$.  
		For any [iso-]morphism $\phi\colon \hat{X}\ra \hat{Y}$ and any $\varepsilon>0$ there exists a [iso-]morphism $\psi\colon X \ra Y  $ such that $||\mcO(\phi)(f)-\mcO(\hat{\psi})(f)||\leq\epsilon$ for all $f\in \mcO(\hat{Y})$ with $||f||\leq1$
	\end{prop}
	
	\begin{rmk}
		Following the notations of the previous proposition, the property $||\mcO(\phi)(f)-\mcO(\hat{\psi})(f)||\leq\epsilon$ for all $f\in \mcO(\hat{Y})$ with $||f||\leq1$ is typically denoted by $||{\phi} - \hat{\psi}||\leq\epsilon$. We will also follow this convention in what follows.
	\end{rmk}
	
	\begin{prop}[{\cite[Paragraph 1.16]{gk-over}}]
		The category of dagger algebras has coproducts $\otimes^\dagger$. The categories $\Aff^\dagger$ and $\AffSm/S $ have fibered products.
	\end{prop}
	
	We recall how coproducts of dagger algebras are formed. Suppose given three dagger algebras $T$,  $R=K\langle\underline{\tau}\rangle^\dagger/I$ and $S=K\langle\underline{\sigma}\rangle^\dagger/J$ and two maps $T\ra R$, $T\ra S$. The dagger algebra $R\otimes^\dagger_TS$ is the image of $K\langle\underline{\tau},\underline{\sigma}\rangle^\dagger$ under the canonical map to the Tate algebra $\hat{R}\widehat{\otimes}_{\hat{T}}\hat{S}$.

	\begin{dfn}
		Let $R $ be	a dagger algebra. We denote by $R \langle\tau_1\,\ldots,\tau_n\rangle^\dagger$ the dagger algebra $R \otimes^\dagger K\langle\tau_1\,\ldots,\tau_n\rangle^\dagger$.
	\end{dfn}
	
	Whenever $f_1,\ldots,f_n,g$ are elements of a Tate algebra $R$ generating the unit ideal, we  denote by $U(f_1,\ldots,f_n/g)$ the rational space of the affinoid variety $\Spa R$ defined by the conditions $|f_i(x)|\leq|g(x)|$.
	
	\begin{prop}[{\cite[Proposition 2.6, Paragraph 2.11]{gk-over}}]\label{eqrational}
		Let $X =\Spa^\dagger R $ be an affinoid dagger space. Any rational open subset $U$ of $\Spa \hat{R}$ can be written as $U(f_1,\ldots,f_n/g)$ with $f_i,g\in R $ generating the unit ideal. The dagger space $U =\Spa^\dagger \mcO^\dagger(U)$ with 
		$$\mcO^\dagger(U)=R \langle\tau_1,\ldots,\tau_n\rangle^\dagger\!/(g\tau_i-f_i)$$ 
		is an open rational subspace of $\Spa^\dagger R $ canonically independent on the choice of $f_i,g$. Moreover $\mcO^\dagger$ is a sheaf of topological $K$-algebras on $\Spa \hat{R}$.
	\end{prop}
	
	\begin{dfn}\label{Odag}
		Let $X $ be an affinoid dagger space with limit $\hat{X}$. We denote by $\mcO^\dagger_{X }$ or simply by $\mcO^\dagger $ the sheaf of topological algebras on the rational site of $X$ as well as the sheaf of topological algebras on $\hat{X}$ introduced in Proposition \ref{eqrational}.
	\end{dfn}

	In the category of affinoid rigid analytic spaces, the functor $\Spa R\mapsto R^\circ$ is represented by $\B^1$. The next proposition shows the role of the dagger disc $\B^{1\dagger}$ introduced above.

	\begin{prop}\label{daggerballrep}
		The  affinoid dagger space $\B^{1\dagger}$  represents the functor $\Spa^\dagger R \mapsto \hat{R}^\circ\cap R $. 
	\end{prop}
	
	\begin{proof}
		Let $X =\Spa^\dagger R$ be an affinoid dagger space. A continuous map $K\langle\tau\rangle\ra \hat{R}$ amounts to the choice of an element $s\in \hat{R}^\circ$ and it preserves the dagger structures if and only if for any $n$ the induced map $K\langle \pi^{1/n}\tau\rangle\ra \hat{R}$  factors over a map $K\langle\pi^{1/n}\tau\rangle\ra \hat{R}_h$ for some $h$, that is,  if and only if $\pi s^n \in \varinjlim \hat{R}_h^\circ$ for all $n$. 
		Since  $R $ is a $K$-algebra, integrally closed in $\hat{R}$ (see \cite[Theorem 2]{BDR}) we deduce that such an $s$ lies in $\hat{R}^\circ\cap R $. 
		
		Vice-versa, we claim that any element $s\in \hat{R}^\circ\cap R^\dagger$ satisfies the condition. 		
		If $f\in \hat{R}_1\cap R^\circ$ then we deduce from \cite[Proposition 2.1.16]{ayoub-rig} the following inclusions in $X_1=\Spa \hat{R}_1$:
		\[
		\hat{X}\subset U(f/1)\Subset_{X_1}U(\pi f/1).
		\]
		Since $\{X_h\}$ is coinitial among the rational subvarieties $W$ of $X_1$ such that $\hat{X }\Subset_{X_1}W$ we deduce in particular that $X_h\subset U(\pi f/1)$ for some $h$ that is,  $\pi f\in\varinjlim \hat{R}_h^\circ$. This proves the inclusions
		\[
		\pi(R \cap \hat{R}^\circ)\subset  \varinjlim \hat{R}_h^\circ\subset R \cap \hat{R}^\circ
		\]
		and therefore our claim. 
	\end{proof}

	The following proposition already appears in \cite[Theorem 2.3]{etesse-rel2}. We present here an alternative proof based on the methods developed in the Appendix.
	
	\begin{prop}\label{dagcov1}
		Let $X $ be 
		an affinoid dagger space with limit $\hat{X}$. 
		\begin{enumerate}
			\item The functor $U \mapsto \hat{U}$ defines an equivalence between the categories of inclusions of rational subspaces in   $X $ and in $\hat{X}$.
			\item The functor $U\mapsto \hat{U}$ defines an equivalence between the categories of finite \'etale affinoid spaces over   $X$ and over $\hat{X}$.
		\end{enumerate}	
	\end{prop}
	
	\begin{proof}
		The first claim follows from Proposition \ref{eqrational}. 
		For the second claim, by \cite[Lemma 7.5]{scholze} we know that, up to shifting indices, any map  $\hat{U}\ra \hat{V}$ of   finite \'etale affinoid spaces over $\hat{X}$ is induced by a map $U_1\ra V_1$ of finite \'etale spaces over $X_1$ with $\hat{U}=U_1\times_{X_1}\hat{X}$ and $V=V_1\times_{X_1}\hat{X}$. Let $U_h$ be $U_1\times_{X_1}X_h$. We are left to prove that $\varinjlim\mcO(U_h)$ is a dagger algebra.

		We now use the equivalence between dagger affinoid spaces and their presentations, proved in the Appendix (see Proposition \ref{moraffdag}). In particular, we can alternatively prove that the sequence $U_h$ is a presentation of $\hat{U}$. It suffices to show that  $\hat{U}$ lies in $\Int(U_1)$ (see the notations of Definition \ref{interiors}).  
		Since $f\colon U_1\ra X_1$ is finite, by \cite[Corollary 2.5.13(i)]{berkovich} and Corollary \ref{berkint} we deduce that $\Int(U_1/X_1)=U_1$ and hence by Corollary \ref{berkform} we get $\Int(U_1)=f^{-1}\Int(X_1)$. We then need to prove that the image of $\hat{U}$ lies in the interior of $X_1$ and this is clear as it factors over $\hat{X}$ lying in $\Int( X_1)$.
	\end{proof}

	\begin{rmk}\label{fet}
		If $Y\ra X $ is a finite \'etale morphism of affinoid dagger spaces, the dagger algebra $\mcO^\dagger(Y )$ associated with $Y $ is of the form $\varinjlim(\hat{S}_1\otimes_{\hat{R}_1}\hat{R}_h)=\hat{S}_1\otimes_{\hat{R}_1}\mcO^\dagger(X )$ for some finite \'etale algebra $\hat{S}_1$ over $\hat{R}_1$ (up to shifting indices). In particular, it is finite \'etale over $\mcO^\dagger(X )$.
	\end{rmk}

	\begin{cor}\label{dagcov}Let $X $ be an affinoid dagger variety with   limit $\hat{X}$. 
		If $\mcV$ is an \'etale cover of $\hat{X}$, then it can be refined into another cover $\hat{\mcU}$ induced by an \'etale cover $\mcU$ of $X$.
	\end{cor}
	
	\begin{proof}
		By Proposition \ref{dagcov1} and the fact that any \'etale cover can be refined into a new one which is a composition of rational embeddings and finite \'etale maps, we can find a refinement $\hat{\mcU}$ of the cover which is the limit of a family $\mcU$ of maps of dagger spaces. This is also a cover of $X$ by definition. 
	\end{proof}

	\begin{cor}\label{eqtop}
		Let $X $ be an affinoid dagger variety with  limit $\hat{X}$. 
		The maps of the small rational and \'etale sites $\hat{X}\ra X $ induces  equivalences on the associated topoi.
	\end{cor}

	\begin{proof}
		It suffices to use the criteria of \cite[Appendix A]{huber}.
	\end{proof}

	\begin{dfn}
		Let $X =\Spa^\dagger R$ be an affinoid dagger space and $M$ be a finite $R $-module. We define $\tilde{M}$ to be the sheaf $\tilde{M}=M\otimes_{R }\mcO^\dagger$ on $\hat{X}$. 
		A \emph{coherent $\mcO^\dagger$-module} over a dagger space $X $ is a $\mcO^\dagger$-module in the category of  sheaves of $K$-algebras over $\hat{X}$ isomorphic to $\tilde{M}$ for some finite $R $-module $M$.
	\end{dfn}
	
	\begin{rmk}
		In \cite[Theorem 2.16]{gk-over} it is proved that the notion of coherent sheaves over $\Spa^\dagger R $  coincides with the notion of coherent modules over the ringed space $(\Spa \hat{R},\mcO^\dagger)$ of \cite[Chapter 0, Section 5.3]{EGAI}.
	\end{rmk}
	
	\begin{prop}\label{cohocoh}
		Let $X $ be an affinoid dagger space and $\mcF$ be a coherent $\mcO^\dagger$-module over it.
		\begin{enumerate}
			\item({\cite[Proposition 3.1]{gk-over}}) $H^i_{\an}(X ,\mcF)=0$ if $i>0$.
			\item $\mcF$ extends to an \'etale sheaf over $X $ defined by putting $f^*\mcF=f^{-1}\mcF\otimes_{f^{-1}\mcO^{\dagger}_X}\mcO^\dagger_Y$ for each \'etale map $Y \ra X $.
			\item $H^i_{\et}(X ,\mcF)=0$ if $i>0$.
		\end{enumerate}
	\end{prop}
	
	\begin{proof}
		Suppose $\mcF=\tilde{M}$ and let $R $ be the dagger algebra $\mcO^\dagger(X )$. By means of Corollary \ref{eqtop}, \cite[Proposition 8.2.1]{fvdp} and the proof of \cite[Proposition 8.2.3(2)]{fvdp} we are left to prove that or each surjective finite \'etale map  $\hat{Y}\ra \hat{X}$ the following sequence is exact 
		$$0\ra M\ra M\otimes_{R }S \rightrightarrows M\otimes_{R }(S \otimes_{R }^\dagger S )$$
		where we denote by  $S $ the dagger algebra associated with $Y$ (see Proposition \ref{dagcov1}). By Remark \ref{fet} the map $R \ra S $ is finite \'etale, and in particular the dagger tensor product coincides with the usual one (see \cite[Lemma 1.10]{gk-over}). The claim then follows from \cite[Section I.3]{FGA}.
	\end{proof}
	
	We now recall the definition of Gro\ss e-Kl\"onne of ``global" dagger spaces
	\begin{dfn}\label{dagsp}
		A \emph{ dagger space} is a pair $X =(\hat{X},\mcO^\dagger)$ where $\hat{X} $ is a rigid analytic space, and $\mcO^\dagger$ is  a sheaf of topological $K$-algebras on $\hat{X}$ such that for some affinoid open cover $\{\hat{U}_i\ra \hat{X}\}$ there are dagger structures $U_i $ on $\hat{U}_i$ with $\mcO^\dagger|_{\hat{U}_i}\cong\mcO^\dagger_{U_i}$ (see Proposition \ref{eqrational}). A morphism of dagger spaces is a morphism of the  underlying locally ringed spaces over $K$  (see \cite[Definition 2.12]{gk-over})  and the category they form is denoted by $\Rig^\dagger$. We say that the rigid space $\hat{X}$ is the \emph{limit} of $X $ and vice-versa we say that $X $ is a \emph{dagger structure} of $\hat{X}$. We say that $X $ [resp. a morphism $X \ra X' $] has the property $\mathbf{P}$ if $\hat{X}$ [resp. the induced morphism of rigid  spaces $\hat{X}\ra\hat{ X}'$] has the property $\mathbf{P}$.  Whenever $S$ is a dagger space, we denote by $\Rig^\dagger$ the category of rigid spaces over it. The category of smooth morphisms of dagger spaces $X \ra S$ to a  fixed dagger space $S $  is denoted by $\RigSm^\dagger\!/S$  and its full subcategory of affinoid objects by $\AffSm^\dagger\!/S$.  
		A collection of morphisms $\{X_i \ra X \}$ in $\RigSm/S $ is a \emph{cover} if the induced collection $\{\hat{X}_i\ra \hat{X}\}$  is a topological cover, that is,  if   $\bigsqcup \hat{X}_i\ra \hat{X}$ is surjective.
	\end{dfn}
	
	\begin{exm}
		Any  rigid variety without boundary has a  dagger structure by \cite[Theorem 2.27]{gk-over}. We  denote by $\mathbb{P}^{1\dagger}$ a  dagger variety having as limit the projective line $\mathbb{P}^1$.
	\end{exm}
	
	The following proposition is straightforward.
	
	\begin{prop} 
		Let $S $ be a  dagger space.  The functor 
		$$
		\begin{aligned}
		\Aff ^\dagger/S&\ra\Rig ^\dagger/S\\
		\Spa^\dagger R&\mapsto(\Spa \hat{R},\mcO^\dagger)
		\end{aligned}
		$$
		is   fully faithful and induces an equivalence of the associated open analytic and the \'etale  topoi.
	\end{prop}

	From now on, we will use the term \emph{affinoid dagger space} also to indicate the objects in the essential image of the functor above. 
	
	We easily obtain also the following version.

	\begin{cor} Let $S $ be a  dagger space. 
		The functor 
		$$
		\begin{aligned}
		\AffSm^\dagger\!/S &\ra\RigSm^\dagger\!/S \\
		\Spa^\dagger R&\mapsto(\Spa \hat{R},\mcO^\dagger)
		\end{aligned}
		$$
		is   fully faithful and induces an equivalence of the associated open analytic and the \'etale  topoi.
	\end{cor}

	\section{Approximation Results}\label{approx}
	
	From now on, we fix a  dagger space $S $ with limit $\hat{S}$. 
	In this section, we recall the analytic version of the  inverse function theorem and we  use it as an alternative to Artin's approximation theorem for smooth dagger algebras \cite{bosch-artin}. As a matter of fact, it induces a weaker form of this theorem (see Corollary \ref{artins}) but also a cubical version of it that we will need in what follows (see Propositions \ref{2dH=02} and \ref{sollevacocu}).
	
	The reader who believes in Proposition \ref{sollevacocu} can safely skip this technical section.
	
	\begin{prop}
		\label{roots3}
		Let $R $ be a  dagger algebra with completion $\hat{R}$. If an element $\xi$ of $\hat{R}$ is algebraic over $\Frac R $ then it lies in $R $.
	\end{prop}
	
	\begin{proof}
		We denote by $X $ the space $\Spa^\dagger R$. Since $R $ is algebraically closed in ${\hat{R}}$ (see \cite[Theorem 2]{BDR}) we conclude that $\Frac R $ is algebraically closed in $( R \setminus\{0\})^{-1}\hat{R}$ and therefore $\xi\in\Frac R $. 
		We can also assume $\xi\in \Frac \hat{R}_1$ up to shifting indices. 
		
		Let $I_{\xi}=(d_1,\ldots,d_n)$ be the ideal of denominators of $\hat{R}_1$ associated with the meromorphic  function $\xi$ (see \cite[Lemma 4.6.5]{fvdp}) and let $V(I_{\xi})$ be the induced (Zariski) closed subvariety of $X_1=\Spa \hat{R}_1$. From now on, we denote by $T^c$ the closure of a subset $T$ in $X_1$. We recall that $\xi$ is analytic around a point $x$ if and only if $(I_{\xi})_x=\mcO_{x}$.   Since $\mcO_{x}$ is local with maximal ideal equal to the support of the valuation at $x$ (see \cite[Lemma 1.6(i)]{huber2}) we deduce that $\xi$ is regular around $x$ if and only if   $x\notin V(I_\xi)$ that is,  if $|d_i(x)|\neq0$ for some $i$. By \cite[Lemma 1.1.10]{huber}
		if $x\in{\{y\}}^c$ for some point $y$ then $|d_i(x)|=0$ if and only if $|d_i(y)|=0$. Using \cite[Remark 2.1(iii)]{huber1} 
		and the regularity of $\xi$ on $\hat{X}$ we 
		then deduce that ${\hat{X}}^c\cap V(I_\xi)=\emptyset$. On the other hand, 
		by \cite[Lemma 1.5.10(1a)]{huber} this  set  ${\hat{X}}^c\cap V(I_\xi)$ coincides with the intersection of the nested closed subsets $\{{X}^c_h\cap V(I_{\xi})\}_h$ inside the quasi-compact space $X_1$. From Cantor's intersection theorem (see e.g. \cite[Theorem 3-5.9]{munkres}) we conclude $X_h\cap V(I_\xi)=\emptyset$  for $h$ large enough and hence  $\xi$ is regular on $X_h$, as wanted.
	\end{proof}
	
	We obtain in particular the following fact.
	
	\begin{cor}\label{invdag}
		Let $R $ be a dagger algebra with completion $\hat{R}$. If $f\in R $ is invertible in $\hat{R}$ then it is invertible in $R $.
	\end{cor}

	We recall the following  version of the inverse mapping theorem in the analytic context.

	\begin{prop}[{{\cite[Corollary A.2]{vezz-fw}}}]\label{implicitC}
		Let $\hat{R}$ be a non-archimedean  Banach  $K$-algebra, let $ {\sigma}=(\sigma_1,\ldots,\sigma_n)$ and $ {\tau}=(\tau_1,\ldots,\tau_m)$ be two systems of coordinates, let $\bar{\sigma}=(\bar{\sigma}_1,\ldots,\bar{\sigma}_n)$ and $ \bar{\tau}=(\bar{\tau}_1,\ldots,\bar{\tau}_m)$ two sequences of elements of $\hat{R}$ and let $ {P}=(P_1,\ldots,P_m)$ be a collection of polynomials in $\hat{R}[ {\sigma}, {\tau}]$ such that $ {P}( {\sigma}=\bar{\sigma}, {\tau}=\bar{\tau})=0$ and $\det(\frac{\del P_i}{\del \tau_j})( {\sigma}=\bar{\sigma}, {\tau}=\bar{\tau})\in \hat{R}^\times$. There exists a unique collection $ {F}=(F_1,\ldots,F_m)$ of $m$ formal power series in $\hat{R}[[ {\sigma-\bar{\sigma}}]]$ such that $ {F}( {\sigma}=\bar{\sigma})=\bar{\tau}$ and $ {P}( {\sigma}, {F}( {\sigma}))=0$ in $\hat{R}[[ {\sigma-\bar{\sigma}}]]$ and they have a positive radius of convergence around $\bar{\sigma}$. 
	\end{prop}

	\begin{cor}\label{artins}
		Suppose that $S$ is affinoid. Let $X $ and $Y $ be two  affinoid dagger algebras smooth over $S $ such that $\hat{Y}$ is \'etale over a poly-disc $\B^m\times \hat{S}$. 
		For any [iso-]morphism $\phi\colon \hat{X}\ra \hat{Y}$ over $\hat{S}$ and any $\varepsilon>0$ there exists a [iso-]morphism $\psi \colon X \ra Y $ over $S $ such that $||\phi-\hat{\psi}||\leq\epsilon$. 
	\end{cor}
	
	\begin{proof}
		We suppose $S=\Spa^\dagger A$, $X=\Spa^\dagger R$ and we denote their limits by  $\hat{S}=\Spa \hat{A}$, $\hat{X}=\Spa \hat{R}$. 
		Note that $\hat{Y}$ is  isomorphic to $\Spa \hat{A}\langle\sigma_1,\ldots,\sigma_m,\tau_1,\ldots,\tau_n\rangle/(P_1,\ldots, P_n)$ such that each $P_i$ is a polynomial in $A[\sigma,\tau]$ and  $\det\left(\frac{\del P}{\del \tau}\right)$  is invertible in $\mcO(\hat{Y})$ by means of \cite[Lemma 1.1.51]{ayoub-rig}.
		
		We first assume that  $Y=\Spa^\dagger  A\langle \underline{\sigma},\underline{\tau}\rangle^\dagger\!/(\underline{P}(\underline{\sigma},\underline{\tau}))$. 
		We also remark that  $\det\left(\frac{\del P}{\del \tau}\right)$  is invertible in $\mcO^\dagger (Y)$ by Corollary \ref{invdag}.
		
		The map $\phi$ is uniquely determined by the association $({\sigma},{\tau})\mapsto(s,t)$ from $\hat{A}\langle \underline{\sigma},\underline{\tau}\rangle/(\underline{P})$ to $\hat{R}$ for an $m$-tuple $s$ and an $n$-tuple $t$ in $\hat{R}$. 	
		By  Corollary \ref{implicitC}   there exists  power series $F=(F_{1},\ldots,F_{m})$ in $\hat{R}[[\sigma-\bar{\sigma}]] $  such that
		\[
		(\sigma,\tau)\mapsto(\tilde{s},F(\tilde{s}))\in \hat{R}
		\]
		defines a new map $\psi$ from $\hat{X}$ to $\hat{Y}$  for any choice of $\tilde{s}\in \hat{R}^\circ\cap R $ such that $\tilde{s}$ is in the convergence radius of $F$ and $F(\tilde{s})$ is in $\hat{R}^\circ$. The field $\Frac(R )(F(\tilde{s}))$ is finite \'etale over $\Frac(R )$. By Proposition \ref{roots3} we deduce that $F(\tilde{s}) $ is in $R $ and therefore $\psi$ is a map from $X $ to $Y $. 
		By the density of $R $ in $\hat{R}$ and the continuity of $F$ we can also assume that the $m$-tuple $\tilde{s}$ induces a map $\psi$ such that $||\phi-\psi||\leq\epsilon$.  
		
		If we take $\hat{X}=\hat{Y}$  we can find an endomorphism $\psi$ of $\hat{X}$ inducing a map $X \ra Y $ such that $||\id-\psi||\leq\epsilon$.  If we take $\epsilon$ sufficiently small, we deduce that $\psi$ is an automorphism of $\hat{X}$ and hence any two dagger structures on it are isomorphic.	In particular, we obtain the general case  of the proposition. 
	\end{proof}
	
	In the previous proof, we also showed the following structure theorem.
	
	\begin{cor}\label{etaleoverball}Suppose $S$ affinoid. 
		Let $\hat{Y}$ be an affinoid rigid space which is \'etale over the poly-disc $\B^{m}\times \hat{S}$. Then it admits a dagger structure $Y $  isomorphic to $\Spa R $ with $$R = \mcO^\dagger (\hat{S})\langle\sigma_1,\ldots,\sigma_m,\tau_1,\ldots,\tau_n\rangle^\dagger\!/(P_1,\ldots,P_n)$$ where each $P_i $ lies in $\mcO^\dagger(\hat{S})[\underline{\sigma},\underline{\tau}]$ and $\det(\frac{\del P_i}{\del\tau_j})$ is invertible in $R $.
	\end{cor}
	
	\begin{prop}\label{smooth}
		Any  rigid dagger space smooth over $S $ is locally \'etale over a dagger poly-disc $\B^{n\dagger}\times{V }$ for some affinoid rational open subset $V\subset S$.
	\end{prop}
	
	\begin{proof}The claim is local on $S $ so we can assume it is affinoid. 
		Let $X $ be a smooth dagger variety over it. By Proposition \ref{dagcov1} and \cite{ayoub-rig} we can find a rational covering $\{U _i\ra X \}$ such that each limit $\hat{U}_i$ is \'etale over some poly-disc $\B^n_{\hat{V}_i}$ where $\hat{V}_i $ is rational inside $\hat{S}$. The claim then follows from Corollary \ref{etaleoverball}.
	\end{proof}

	We recall (see \cite[Definition 1.1.9/1]{BGR}) that a morphism of normed groups $\phi\colon G\ra H$ is \emph{strict} if the homomorphism $G/\ker\phi\ra\phi(G)$ is a homeomorphism, where the former group is endowed with the quotient topology and the latter with the topology inherited from $H$. In particular, we say that a sequence of normed $K$-vector spaces
	\[
	R\stackrel{f}{\ra} R'\stackrel{g}{\ra} R''
	\]
	is \emph{strict and exact} at $R'$ if it exact at $R'$ and if $f$ is strict that is,  the quotient norm and the norm induced by $R'$ on $R/\ker(f)\cong\ker(g)$ are equivalent.
	
	\begin{lemma}\label{chinese22}
		For any  map $\sigma\colon T_\sigma\ra\{0,1\}$ defined on a subset $T_\sigma$ of $\{1,\ldots,n\}$ we denote by $I_\sigma$ the ideal 
		generated by $\theta_i-\sigma(i)$ as $i$ varies in $T_\sigma$. 
		For any finite  set $\Sigma$ of such maps and any dagger algebra $R$ with limit $\hat{R}$ the following diagram  of topological $K$-algebras has vertical inclusions and strict and exact lines 
		$$\xymatrix{
			0\ar[r] &\hat{R} \langle  \underline{\theta}\rangle/\displaystyle\bigcap_{\sigma\in\Sigma} I_\sigma\ar[r]&\displaystyle\prod_{\sigma\in\Sigma} \hat{R} \langle  \underline{\theta}\rangle/I_{\sigma}\ar[r]&\displaystyle\prod_{\sigma,\sigma'\in\Sigma} \hat{R}\langle  \underline{\theta}\rangle/(I_\sigma+I_{\sigma'})\\
			0\ar[r] &R\langle\underline{\theta}\rangle^\dagger\!/\displaystyle\bigcap_{\sigma\in\Sigma} I_\sigma\ar[r]\ar@{^{(}->}[u]&\displaystyle\prod_{\sigma\in\Sigma} R\langle\underline{\theta}\rangle^\dagger\!/I_{\sigma}\ar[r]\ar@{^{(}->}[u]&\displaystyle\prod_{\sigma,\sigma'\in\Sigma} R\langle\underline{\theta}\rangle^\dagger\!/(I_\sigma+I_{\sigma'})\ar@{^{(}->}[u]
		}$$
		Moreover, the ideal $\bigcap_{\sigma\in\Sigma} I_\sigma$ is generated by a finite set of polynomials with coefficients in $\Z$.
	\end{lemma}

	\begin{proof}
		The fact that the first line is strict and exact as well as the description of the generators of $\bigcap I_\sigma$ is proved in \cite{vezz-fw}. The same statement applied to the rings $\hat{R}_h\langle\pi^{1/h}\underline{\theta}\rangle$ and a direct limit argument show that also the second line is exact. We now prove that the vertical maps are inclusions. This can be proved only for the last two columns, where the statement is clear. As the second line is isometrically contained in the first, we also deduce that it is strict as well.
	\end{proof}

	Let $\sigma$ and $\sigma'$ be maps defined from two subsets $T_\sigma$ resp. $T_{\sigma'}$ of $\{1,\ldots,n\}$ to $\{0,1\}$. We say that they are \emph{compatible} if $\sigma(i)=\sigma'(i)$ for all $i\in T_\sigma\cap T_{\sigma'}$ and in this case we denote by $(\sigma,\sigma')$ the map from $T_\sigma\cup T_{\sigma'}$ extending them.
	
	\begin{lemma}\label{chinesecor2b}
		Let $R $ be a dagger algebra with completion $\hat{R}$ and $\Sigma$ a set as in Lemma \ref{chinese22}. 
		For any $\sigma\in\Sigma$ let  $\bar{f}_\sigma$ be an element of $\hat{R}\langle\underline{\theta}\rangle/I_\sigma$ such that $\bar{f}_\sigma|_{(\sigma,\sigma')}=\bar{f}_{\sigma'}|_{(\sigma,\sigma')}$ for any couple $\sigma,\sigma'\in\Sigma$ of compatible maps. 
		\begin{enumerate}
			\item There exists an element $f\in \hat{R}\langle\underline{\theta}\rangle$ such that $f|_\sigma=\bar{f}_\sigma$. 
			\item There exists a constant $C=C(\Sigma)$ such that if for some $g\in \hat{R}\langle\underline{\theta}\rangle$ one has $|\bar{f}_\sigma-{g}|_\sigma|<\varepsilon$ 
			for all $\sigma$ then the element $f$ can be chosen so that $|f-g|<C\varepsilon$. Moreover, if $\bar{f}_\sigma\in R\langle\underline{\theta}\rangle^\dagger\!/I_\sigma$ for all $\sigma$ then the element $f$ can be chosen inside $R\langle\underline{\theta}\rangle^\dagger$. 
		\end{enumerate}
	\end{lemma}
	
	\begin{proof}
		The first claim and the first part of the second are simply a restatement of Lemma \ref{chinese22}, where $C=C(\Sigma)$ is  the constant defining the compatibility $||\cdot||_1\leq C||\cdot||_2$ between the  norm $||\cdot||_1$ on $\hat{R}\langle\underline{\theta}\rangle/\bigcap I_\sigma$ induced by the quotient and the norm $||\cdot||_2$ induced by the embedding in $\prod \hat{R}\langle\underline{\theta}\rangle/I_\sigma$. We now turn to the last sentence of the second claim.

		By Lemma \ref{chinese22} and what proved above there exist two lifts of $\{\bar{f}_\sigma\}$: an element $f_1$ of $R\langle\underline{\theta}\rangle^\dagger$ and an element $f_2$ of $\hat{R}\langle\underline{\theta}\rangle$ such that $|f_2-g|<C\varepsilon$ and their difference lies in $\bigcap I_\sigma$. Hence, we can find elements $\gamma_i\in \hat{R}\langle \underline{\theta}\rangle$ such that $f_1=f_2+\sum_i\gamma_i p_i$ where $\{p_1,\ldots,p_M\}$ are generators of $\bigcap I_\sigma$ which have coefficients in $\Z$. Let now $\tilde{\gamma}_i$ be elements of $R\langle\underline{\theta}\rangle^\dagger$ with $|\tilde{\gamma}_i-\gamma_i|<C\varepsilon/M|p_i|$. The element $f_3\colonequals f_1-\sum_i\tilde{\gamma}_i p_i$ lying in $R\langle\underline{\theta}\rangle^\dagger$ is another lift of $\{\bar{f}_\sigma\}$ and satisfies  $|f_3-g|\leq\max \{|f_2-g|,|f_2-f_3|\}< C\varepsilon$ proving the claim.
	\end{proof}

	\begin{prop}\label{2dH=02}
		Let $R $ be a dagger algebra with completion $\hat{R}$. Let $s_1,\ldots,s_N$ be elements of $\hat{R}\langle\theta_1,\ldots,\theta_n\rangle^\circ$. 
		For any $\varepsilon>0$ there exist elements $\tilde{s}_1,\ldots,\tilde{s}_N$ of $\hat{R}\langle\underline{\theta}\rangle^\circ\cap R\langle\underline{\theta}\rangle^\dagger$
		satisfying the following conditions.
		\begin{enumerate}
			\item $|s_\alpha-\tilde{s}_\alpha|<\varepsilon$ for each $\alpha$.
			\item For any $\alpha,\beta\in\{1,\ldots,N\}$ and any $k\in\{1,\ldots,n\}$ such that $s_\alpha|_{\theta_k=0}=s_\beta|_{\theta_k=0}$ we also have $\tilde{s}_\alpha|_{\theta_k=0}=\tilde{s}_\beta|_{\theta_k=0}$.
			\item For any $\alpha,\beta\in\{1,\ldots,N\}$ and any $k\in\{1,\ldots,n\}$ such that $s_\alpha|_{\theta_k=1}=s_\beta|_{\theta_k=1}$ we also have $\tilde{s}_\alpha|_{\theta_k=1}=\tilde{s}_\beta|_{\theta_k=1}$.
			\item\label{cond112ex} For any $\alpha\in\{1,\ldots,N\}$ if $s_\alpha|_{\theta_1=1}\in R\langle\underline{\theta}\rangle^\dagger$  then $\tilde{s}_\alpha|_{\theta_1=1}={s}_\alpha|_{\theta_1=1}$.
		\end{enumerate}
	\end{prop}
	
	\begin{proof}
		We will actually prove a stronger statement, namely that we can reinforce the previous conditions with the following:
		\begin{enumerate}\setcounter{enumi}{4}
			\item\label{condx}  For any $\alpha,\beta\in\{1,\ldots,N\}$ any subset $T$ of $\{1,\ldots,n\}$ and any map $\sigma\colon T\ra\{0,1\}$ such that $s_\alpha|_{\sigma}=s_\beta|_{\sigma}$ then $\tilde{s}_\alpha|_{\sigma}=\tilde{s}_\beta|_{\sigma}$.
			\item For any $\alpha\in\{1,\ldots,N\}$ any subset $T$ of $\{1,\ldots,n\}$ containing $1$ and any map $\sigma\colon T\ra\{0,1\}$ such that $s_\alpha|_{\sigma}\in R\langle\underline{\theta}\rangle^\dagger$  then $\tilde{s}_\alpha|_{\sigma}={s}_\alpha|_{\sigma}$.
		\end{enumerate}
		Above we denote by $s|_\sigma$ the image of $s$ via the substitution  $(\theta_{t}=\sigma(t))_{t\in T}$. 
		We proceed by induction on $N$, the case $N=0$ being trivial. We remark that if $\epsilon$ is sufficiently small, any element $a$ such that $|a-s_k|<\varepsilon$ lie in  $\hat{R}\langle\underline{\theta}\rangle^\circ$  as this ring is open. We are left to prove that we can pick elements $\tilde{s}_k$ in  $R\langle\underline{\theta}\rangle^\dagger$.
		
		Consider the conditions we want to preserve that involve the index $N$. They are of the form
		\[
		{s}_i|_\sigma={s}_N|_{\sigma}
		\]
		and are indexed by some pairs $(\sigma, i)$ where $i$ is an index and  $\sigma$ varies in a set of maps $\Sigma$. Our procedure consists in determining by induction the elements $\tilde{s}_1,\ldots,\tilde{s}_{N-1}$ first, and then deduce the existence of $\tilde{s}_N$ by means of Lemma \ref{chinesecor2b} by lifting the elements $\{\tilde{s}_i|_{\sigma}\}_{(\sigma,i)}$. 
		Therefore, we first define $\varepsilon'\colonequals\frac{1}{C}\varepsilon$ where $C=C(\Sigma)$ is the constant introduced in Lemma \ref{chinesecor2b} and then apply the induction hypothesis to the first $N-1$ elements with respect to $\varepsilon'$.
		
		By the induction hypothesis, the elements $\tilde{s}_i|_\sigma$ satisfy the compatibility condition of Lemma \ref{chinesecor2b} and lie in $R\langle\underline{\theta}\rangle^\dagger$. 
		By Lemma \ref{chinesecor2b} we can find an element $\tilde{s}_N$ of $R\langle\underline{\theta}\rangle^\dagger$ lifting them 
		such that $|\tilde{s}_N-s_N|<C\varepsilon'=\varepsilon$ as wanted.
	\end{proof}
	
	We now restate Proposition \ref{2dH=02} in more geometric terms.

	\begin{prop}
		\label{sollevacocu}Suppose that $S$ is affinoid.                                                     Let 
		$X $ be an affinoid dagger space smooth over $S $ and 
		let $Y $ be an affinoid dagger algebra such that $\hat{Y}$ is \'etale over $\B^{m}\times \hat{S}$. 
		For a given finite set of maps $\{f_1,\ldots,f_N\}$ in $\Hom_{\hat{S}}(\hat{X}\times\B^{n},\hat{Y})$ we can find corresponding maps $\{H_1,\ldots,H_N\}$ in
		$\Hom_{\hat{S}}(\hat{X}\times\B^{n}\times\B^{1}, \hat{Y})$  
		such that:
		\begin{enumerate}
			\item For all $1\leq k\leq N$ it holds $i_0^*H_k=f_k$ and $i_1^*H_k$ lies in $\Hom_{S }(X \times\B^{n\dagger},Y )$. 
			\item If ${f}_k\circ d_{r,\epsilon}={f}_{k'}\circ d_{r,\epsilon}$ for some $1\leq k,k'\leq N$ and some $(r,\epsilon)\in\{1,\ldots,n\}\times\{0,1\}$ then $H_k\circ d_{r,\epsilon}=H_{k'}\circ d_{r,\epsilon}$.
			\item If for some $1\leq k\leq N$ and some $h\in\N$ the map $f_k\circ d_{1,1}\in \Hom_{\hat{S}}(\hat{X}\times\B^{n-1},\hat{Y})$  lies in $\Hom_{S }(X \times\B^{(n-1)\dagger},Y )$ then the element $H_k\circ d_{1,1}$ of $\Hom_{\hat{S}}(\hat{X}\times\B^{n-1}\times\B^1,\hat{Y})$ is constant on $\B^1$  equal to $f_k\circ d_{1,1}$.
		\end{enumerate}
	\end{prop}

	\begin{proof}
		We let $S $ be $\Spa A $ with limit $\hat{S}=\Spa \hat{A}$. By Corollary \ref{etaleoverball} we can assume that $Y =\Spa A\langle \sigma_1,\ldots,\sigma_n,\tau_1,\ldots,\tau_m\rangle^\dagger\!/(P_i(\underline{\sigma},\underline{\tau}))$ for $m$ polynomials $P_i\in A [\underline{\sigma},\underline{\tau}]$ such that $\det\left(\frac{\del P}{\del \tau}\right)$ is invertible in $\mcO^\dagger(\hat{Y}) $.

		For any $h\in\Z$ we  denote by $R$ the dagger algebra $\mcO^\dagger(\hat{X})\langle\underline{\theta}\rangle^\dagger$ %
		and by $R\langle\chi\rangle^\dagger$ the dagger algebra associated with $X \times\B^{n\dagger}\times\B^{1\dagger}$. 
		Each $f_k$ is induced by maps $(\underline{\sigma},\underline{\tau})\mapsto({s}_k,{t}_k)$ from $K\langle \underline{\sigma},\underline{\tau}\rangle/(\underline{P})$ to $\hat{R}$ for some $m$-tuples ${s}_k$ and $n$-tuples ${t}_k$ in $R$. 	
		By  Corollary \ref{implicitC}   there exists a sequence of power series $F_k=(F_{k1},\ldots,F_{km})$ associated with each $f_k$ such that
		\[
		(\sigma,\tau)\mapsto(s_k+(\tilde{s}_k-s_k)\chi,F_k(s_k+(\tilde{s}_k-s_k)\chi))\in \hat{R}\langle\chi\rangle\cong\mcO(\hat{X}\times\B^n\times\B^1)
		\]
		defines a map $H_k$ from $\hat{X}\times\B^n\times\B^1$ to $\hat{Y}$  for any choice of $\tilde{s}_k\in \hat{R}^\circ\cap R $ such that $\tilde{s}_k$ is in the convergence radius of $F_k$ and $F_k(\tilde{s}_k)$ is in $\hat{R}^\circ$. 
		
		We prove that any such map satisfies the first claim. By Proposition \ref{daggerballrep} this amounts to prove that the elements $\tilde{t}_k\colonequals F_k(\tilde{s})$ lie in $\hat{R}^\circ\cap R$. By our choice of $\tilde{s}_k$ we already know that they lie in $\hat{R}^\circ$.  We remark that the field $(\Frac R )(\tilde{t}_k)$ is finite \'etale over $\Frac R $. By Lemma \ref{roots3} we deduce that each $\tilde{t}_k$ lies in $R $ as wanted.

		Let now $\varepsilon$ be a positive real number, smaller than all radii of convergence of the series $F_{kj}$ and such that $F(a)\in R $ for all $|a-s|<\varepsilon$. 
		Denote by $\tilde{s}_{ki}$ the elements associated with $s_{ki}$ by applying Proposition \ref{2dH=02} with respect to the chosen $\varepsilon$. In particular, they induce a well-defined map $H_k$  and the elements $\tilde{s}_{ki}$ lie in $\hat{R}^\circ\cap R $. We show that the maps $H_k$ induced by this choice also satisfy the second, third and fourth claims of the proposition.

		Suppose that $f_k\circ d_{r,\epsilon}=f_{k'}\circ d_{r,\epsilon}$ for some $r\in\{1,\ldots,n\}$ and $\epsilon\in\{0,1\}$. This means that $\bar{s}\colonequals s_{k}|_{\theta_r=\epsilon}=s_{k'}|_{\theta_r=\epsilon}$ and $\bar{t}\colonequals t_{k}|_{\theta_r=\epsilon}=t_{k'}|_{\theta_r=\epsilon}$. 
		This implies that both $F_k|_{\theta_r=\epsilon}$ and $F_{k'}|_{\theta_r=\epsilon}$ are two $m$-tuples of formal power series $\bar{F}$ with coefficients in $\mcO(\hat{X}\times\B^{n-1})$ converging around $\bar{s}$ and such that $P(\sigma,\bar{F}(\sigma))=0$, $\bar{F}(\bar{s})=\bar{t}$. By the uniqueness of such power series stated in Corollary \ref{implicitC}, we conclude that they coincide.
		
		Moreover, by our choice of the elements $\tilde{s}_k$ it follows that $\bar{\tilde{s}}\colonequals\tilde{s}_{k}|_{\theta_r=\epsilon}=\tilde{s}_{k'}|_{\theta_r=\epsilon}$. In particular  one has  
		\[
		F_k((\tilde{s}_{k}-s_k)\chi)|_{\theta_r=\epsilon}=\bar{F}((\bar{\tilde{s}}-\bar{s})\chi)=F_{k'}((\tilde{s}_{k'}-s_{k'})\chi)|_{\theta_r=\epsilon}
		\]
		and therefore $H_k\circ d_{r,\epsilon}=H_{k'}\circ d_{r,\epsilon}$ proving the second claim.
		
		The third claim follows from the fact that the elements $\tilde{s}_{ki}$ satisfy the condition (\ref{cond112ex}) of Proposition \ref{2dH=02}.
	\end{proof}

	\section{Dagger Rigid Motives}\label{motives}
	Motives (here, a short form for mixed, derived, effective motives with coefficients in a ring $\Lambda$)  can be defined out of an arbitrary site with products $(\cat,\tau)$ and a choice of a ``contractible object" $I$ inside it. The underlying idea is to construct the universal  category with respect to derived $\Lambda$-functors defined on $\cat$, which   satisfy $\tau$-descent and invariance with respect to the contractible object $I$. Typically, such functors are related to (co-)homology theories. Not surprisingly, the category of motives is simply a Verdier quotient of the derived category of $\tau$-sheaves with values in $\Lambda$-modules $\catD(\Sh_\tau(\cat,\Lambda))$ obtained by imposing the condition that all projections $X\times I\ra X$ are invertible. For the reader familiar with the theory of Voevodsky's motives, we remark that we make no (explicit) use of correspondences, and hence our categories of motives are  \emph{without transfers} (but see  Remark   \ref{remayoubgen}).
	
	Hereunder, we introduce the category of motives $\RigDA^{\dagger\eff}_{\et}(S ,\Lambda)$ associated with the \'etale site on smooth dagger varieties over $S $, inspired by the construction of the motives $\RigDA^{\eff}_{\et}(\hat{S},\Lambda)$ associated with the \'etale site on smooth rigid analytic varieties (see \cite{ayoub-rig}). 
	In the dagger case, the ``contractible object" is the  dagger disc $\B^{1\dagger}$  while in the rigid analytic setting it is the closed disc $\B^1$.

	For computations, and more crucially to invert the Tate twist, the description of motives as Verdier quotients is sometimes unsatisfactory, and specific ``models'' of these triangulated categories are needed to have control on  Hom-sets. This is why the language of model categories is used, applied to the categories of complexes of presheaves. We borrow all the notations from Ayoub (see  \cite{ayoub-rig} and \cite{ayoub-th2}) and we refer to his survey \cite{ayoub-icm} for a gentle explanation of these constructions in the algebraic context. 
	We also refer to \cite{vezz-fw} for a brief collection of some standard results about the change of models, which are based on the results of \cite{ayoub-th2} and inspired by  the classic papers \cite{jardine-s},  \cite{mvw} and  \cite{mv-99}.

	The reader who is not interested in these formal constructions can   skip the technical statements of the section, focusing only on its main result which is Theorem \ref{main}. It is  formally obtained from Proposition \ref{sollevacocu}, proved in the previous section.
	
	From now on, we fix a commutative ring $\Lambda$ and work  with $\Lambda$-enriched categories. In particular, the term ``presheaf'' should be understood as ``presheaf of $\Lambda$-modules'' and similarly for the tem ``sheaf''. The presheaf $\Lambda(X)$ represented by an object $X$ of a category $\cat$  sends an object $Y$ of $\cat$ to the free $\Lambda$-module $\Lambda\Hom_{\cat}(Y,X)$. 
	
	The category $\Ch(\Psh(\cat))$ of complexes of presheaves over a category $\cat$ can be endowed with the \emph{projective model structure} for which weak equivalences are quasi-isomorphisms and fibrations are maps $\mcF\ra\mcF'$ such that $\mcF(X)\ra\mcF'(X)$ is a surjection for all $X$ in $\cat$ (cfr \cite[Section 2.3]{hovey} and \cite[Proposition 4.4.16]{ayoub-th2}). Its homotopy category is the usual (unbounded) derived category $\catD(\Psh(\cat))$.

	\begin{dfn}
		We denote by $\mathcal{S}_{\et}$ the class of   maps $\mcF\ra\mcF'$ in $\Ch\Psh(\RigSm^\dagger\!/S)$ inducing isomorphisms on the $\et$-sheaves associated with $H_i(\mcF)$ and $H_i(\mcF')$ for all $i\in\Z$. We denote by $\mathcal{S}_{\B^{1\dagger}}$ the set of all maps $\Lambda(\B^{1\dagger}\times X )[i]\ra\Lambda( X )[i]$   as $X $ varies in   $\RigSm^\dagger\!/S $ and $i$ varies in $\Z$. We denote by $\mathcal{S}_{(\et,\B^{1\dagger})}$ the union of these two classes. For each $\eta\in\{\et,\B^{1\dagger},(\et,\B^{1\dagger})\}$ we let $\Ch_{\eta}\Psh (\RigSm^\dagger\!/S )$ be the left Bousfield localization of the projective model category on $\Ch\Psh (\RigSm^\dagger\!/S )$ with respect to the class $\mathcal{S}_\eta$.

		The homotopy category of $\Ch_{\et,\B^{1\dagger}}\Psh (\RigSm^\dagger\!/S )$ will be denoted by   $\RigDA^{\dagger\eff}(S ,\Lambda)$ and its element $\Lambda(X )$ will be called the \emph{motive} of $X $ for any dagger space $X $ in $\RigSm/S $.  In case $ S=\Spa K$ then we simply write   $\RigDA^{\dagger\eff}(K,\Lambda)$. 
	\end{dfn}

	\begin{prop}[{\cite[Proposition 3.3]{vezz-fw}}]\label{loctop}
		The localizations introduced above are well defined. Moreover,  the  category of $\Ch_{\et}\Psh (\RigSm^\dagger\!/S )$ is Quillen-equivalent to the (unbounded) derived category of $\et$-sheaves   $\catD(\Sh(\RigSm^\dagger\!/S ))$. It is also Quillen-equivalent to the localization over $\mathcal{S}_{\et}$ of the projective structure on $\Ch\Psh (\cat)$ for any full subcategory $\cat$ generating the same \'etale topos. 
	\end{prop}
	
	\begin{rmk}
		In particular $\Ch_{\et,\B^{1\dagger}}\Psh (\RigSm^\dagger\!/S )$    is Quillen equivalent to the localization over $\mathcal{S}_{(\et,\B^{1\dagger})}$ of $\Ch\Psh (\AffSm^\dagger\!/S )$.
	\end{rmk}
	
	In \cite{ayoub-rig} the category $\RigDA^{\eff}_{\et}(\hat{S},\Lambda)$ is introduced. It can be defined as the homotopy category of the localization of $\Ch\Psh(\RigSm/\hat{S})$ over the class $\mathcal{S}_{\et}$ defined before (in the context of rigid analytic spaces) and over the class $\mathcal{S}_{\B^1}$ containing all maps $\Lambda(X\times\B^1)[i]\ra\Lambda(X)[i]$. It is a monoidal category such that the tensor product $\Lambda(X)\otimes\Lambda(Y)$ of two motives associated with varieties $X$ and $Y$ coincides with $\Lambda(X\times_{\hat{S}} Y)$ and the unit object is $\Lambda(\hat{S})$. The same is true for the dagger counterpart.
	
	\begin{prop}[{\cite[Propositions 4.2.76 and 4.4.63]{ayoub-th2}}]
		The category $\RigDA^{\dagger\eff}_{\et}(S ,\Lambda)$ is monoidal. The tensor product $\Lambda(X )\otimes\Lambda(Y )$ of two motives associated with varieties $X $ and $Y $ coincides with $\Lambda(X  \times_{S } Y )$ and the unit object is $\Lambda(S )$.
	\end{prop}

	Motives have been introduced as quotients of the derived categories of presheaves. On the other hand, the canonical quotient functor admits  a fully faithful right adjoint. Therefore, motives can equally be defined as triangulated subcategories of the derived categories of presheaves, which is particularly useful for computing their $\Hom$-sets. We now investigate better this point of view (see \cite{bv-dg}).
	
	\begin{dfn}
		For $\eta\in\{\et, \B^{1\dagger},(\et,\B^{1\dagger})\}$ we say that a map   in    $\Ch\Psh (\RigSm^\dagger\!/S )$ is a \emph{$\eta$-weak equivalence} if it is a weak equivalence in the model structure     $\Ch_{\eta}\Psh (\RigSm^\dagger\!/S )$.

		We say that an object $\mcF$ of the derived category $\catD=\catD(\Psh (\RigSm^\dagger\!/S ))$ is \emph{$\eta$-local} if the functor $\Hom_{\catD}(\cdot,\mcF)$ sends maps in $S_\eta$ to isomorphisms. This amounts to say that $\mcF$ is  quasi-isomorphic to a $\eta$-fibrant object. We use the same terminology for the model categories on $\Ch\Psh (\cat)$ for any full subcategory $\cat$ generating the same \'etale topos. 
	\end{dfn}
	
	\begin{rmk}
		The existence of these  localizations at the level of model categories is granted by the results of Hirschhorn \cite{hirschhorn} used in the references above. At the level of the homotopy categories, using the language of \cite{bv-dg}, these localizations  induce endofunctors $C^\eta$ of  
		$\catD(\Psh(\RigSm^\dagger\!/S ))$ 
		such that $C^\eta\mcF$ is $\eta$-local for all $\mcF$ and there is a natural transformation $C^\eta\ra\id$ which is a pointwise $\eta$-weak equivalence. 
		The  functor $C^\eta$ restricts to a triangulated equivalence on the objects $\mcF $ that are $\eta$-local  
		and one can compute the Hom set $\Hom(\mcF,\mcF')$ in  the homotopy category of the $\eta$-localization  as ${\catD}(\mcF,C^\eta\mcF')$. The same is true for the category $\catD(\Psh(\cat))$ for any subcategory $\cat$ generating the same \'etale topos.
	\end{rmk}

	\begin{rmk}\label{Cet}
		By means of \cite[Proposition 4.4.59]{ayoub-th2} the complex $C^{\et}\mcF$ is such that $$\mathbf{D}(\Lambda(X )[-i],C^{\et}\mcF)=\HH_{\et}^i(X ,\mcF)$$ for all $X $ in $\RigSm^\dagger\!/S $ and all integers $i$. This property characterizes $C^{\et}\mcF$ up to quasi-isomorphisms. We remark that there is an explicit construction of $C^{\et}\mcF$ using the Godement resolution (see \cite[Paragraph 1.11]{bv-dg}).
	\end{rmk}

	There is also a  characterization of the $\B^{1}$-localization, as described in \cite{ayoub-rig} and \cite{vezz-fw}. Such descriptions admit a natural dagger analogue, that we now describe.   This is based on the fact that  $\B^{1\dagger}$ is an interval object (see \cite{riou}).

	\begin{rmk}
		Let $\B^1_h$ be the open $U(\pi\chi^k/1)=U((\pi\chi)^k/\pi^{k-1})$ in $\Spa K\langle\pi\chi\rangle$. 	The dagger variety $ \B^{1\dagger}$ is an interval object with respect to the maps $i_0$ and $i_1$ induced by the points $\chi\mapsto 0$ and $\chi\mapsto1$ respectively, and the multiplication induced by the  multiplication on the limit $\B^1 $. Indeed, the map $\B^1\times\B^1\ra\B^1\ra \B^1_h$ factors over $\B^1_{2h}\times \B^1_{2h}$.
	\end{rmk}

	In the following part, we have opted for the cubical rather than the simplicial approach in order to have clearer computations, and a stronger parallel with the perfectoid case \cite{vezz-fw}.
	
	\begin{dfn}\label{cocu}
		We denote by $\square^\dagger$ %
		the $\Sigma$-enriched cocubical object (see \cite[Appendix A]{ayoub-h1}) defined by putting $\square^{\dagger n}=\B^{n\dagger }\times S=\Spa K\langle \tau_1,\ldots,\tau_n\rangle^\dagger\times S $ %
		and considering the morphisms $d_{r,\epsilon}$ induced by the maps $\B^{ n\dagger}\ra\B^{(n+1)\dagger}$ corresponding to the substitution $\tau_r=\epsilon$ for $\epsilon\in\{0,1\}$ and the morphisms $p_r$ induced by the projections $\B^{n\dagger}\ra\B^{(n-1)\dagger}$. %
		For any  dagger variety $X $ %
		and any presheaf $\mcF$ of dagger varieties [resp. rigid varieties] with values in an abelian category, we can therefore consider the $\Sigma$-enriched cubical object  $\mcF(X \times_{S}\square^\dagger)$ %
		(see \cite[Appendix A]{ayoub-h1}). Associated to any $\Sigma$-enriched cubical object $\mcF$ there are the following complexes: the complex $C^\sharp_\bullet\mcF$ defined as $C^\sharp_n\mcF=\mcF_n$ and with differential $\sum (-1)^r (d_{r,1}^*-d_{r,0}^*)$; the \emph{simple complex} $C_\bullet\mcF$ defined as $C_n\mcF=\bigcap_{r=1}^n\ker d_{r,0}^*$ and with differential $\sum (-1)^r d_{r,1}^*$; the \emph{normalized complex} $N_\bullet\mcF$ defined as $N_n\mcF=C_n\cap\mcF\bigcap_{r=2}^n\ker d_{r,1}^*$ and with differential $-d_{1,1}^*$. By \cite[Lemma A.3, Proposition A.8, Proposition A.11]{ayoub-h2}, the inclusion $N_\bullet\mcF\hookrightarrow C_\bullet\mcF$ is a quasi-isomorphism and both inclusions $C_\bullet\mcF\hookrightarrow C^\sharp_\bullet\mcF$ and $N_\bullet\mcF\hookrightarrow C_\bullet\mcF$ split. For any complex of presheaves $\mcF$ of dagger varieties %
		we let $\Sing^{\B^{1\dagger}}\mcF$  
		be the total complex of the simple complex associated with   $\uhom(\Lambda(\square^\dagger),\mcF)$.  
		It sends the object $X$ to the total complex of the simple complex associated with $\mcF(X\times_{S }\square^\dagger)$. 
	\end{dfn}
	
	We now show that the complex $\Sing^{\B^{1\dagger}}\mcF$ defined above gives rise to the ``universal'' homotopy-invariant cohomology theory  attached to $\mcF$.

	\begin{prop}\label{sing}
		Let $\mcF$ be a complex in $\Ch\Psh(\AffSm^\dagger\!/S )$. 
		\begin{enumerate}
			\item $\Sing^{\B^{1\dagger}}\mcF$ is $\B^{1\dagger}$-local and $\B^{1\dagger}$-weak equivalent to $\mcF$. %
			\item $\Sing^{\B^{1\dagger}}C^{\et}\mcF$ is $(\et,\B^{1\dagger})$-local and $(\et,\B^{1\dagger})$-weak equivalent to $\mcF$.
		\end{enumerate}
	\end{prop}
	
	\begin{proof}
		From \cite[Corollary 1.2.19]{ayoub-rig} and Corollary \ref{eqtop} we obtain that the \'etale cohomological dimension of any affinoid dagger variety $X $ is finite. Following \cite[Proposition 3.10]{vezz-fw} we can then conclude  that the $\et$-localization coincides with the localization over the set $\{\Lambda(\mcU _\bullet)[i]\ra\Lambda(X )[i]\}$  as $\mcU_\bullet \ra X $ varies among
		bounded \'etale hypercoverings of the objects $X $ in $\AffSm/S $ and $i$ varies in $\Z$. The result then follows in the same way as \cite[Proposition 3.15 and Corollary 3.16]{vezz-fw}.
	\end{proof}
	
	\begin{prop}\label{6f}
		Let $f\colon S \ra T $ be a morphism of dagger spaces.
		\begin{enumerate}
			\item The Kan extension of the functor 
			$$
			\begin{aligned}
			f^*\colon\RigSm^\dagger\!/T &\ra\RigSm^\dagger\!/S \\
			(X \ra T )&\mapsto (X \times_{T }S \ra S )	
			\end{aligned}$$
			induces a Quillen pair
			$$
			\adj{f^*}{\Ch_{\et,\B^{1\dagger}}\Psh(\RigSm^\dagger\!/T )}{\Ch_{\et,\B^{1\dagger}}\Psh(\RigSm^\dagger\!/S )}{f_*}
			$$
			such that the functor $\LL f^*$ is monoidal.
			\item If $f$ is smooth, the Kan extension of the functor
			$$
			\begin{aligned}
			f_\sharp\colon	\RigSm^\dagger\!/S &\ra\RigSm^\dagger\!/T \\
			(X \ra S )&\mapsto(X \ra T )
			\end{aligned}$$
			induces a Quillen pair
			$$
			\adj{f_\sharp}{\Ch_{\et,\B^{1\dagger}}\Psh(\RigSm^\dagger\!/S )}{\Ch_{\et,\B^{1\dagger}}\Psh(\RigSm^\dagger\!/T )}{f^*}
			$$
		\end{enumerate}
	\end{prop}
	
	\begin{proof}
		The statement is a formal consequence of the continuity of the two functors, together with \cite[Proposition 4.4.61]{ayoub-th2} and the formulas $f^*(X \times_{T } Y )\cong f^*(X )\times_{S }f^*(Y )$ and $f_\sharp(\B^{1\dagger}\times X )\cong\B^{1\dagger}\times X $.
	\end{proof}

	We are now interested in finding 
	a convenient set of compact objects which generate the  categories above, as triangulated categories with small sums. This will simplify many definitions and proofs in what follows. We first briefly recall the notion of compactness in triangulated categories.

	\begin{dfn}\label{cpt}
		An object $X$ of a triangulated category with small sums $\catT$ is \emph{compact} if for any small collection $\{Y_i\}$ of objects in $\catT$ one has  \[\Hom(X,\bigoplus Y_i)\cong\bigoplus\Hom(X,Y_i).\]
	\end{dfn}
	
	\begin{exm}
		If $R$ is a ring, compact objects in $\catD(R)$ are complexes which are quasi isomorphic to bounded complexes of finite projective $R$-modules (see e.g. \cite[Tag 07LT]{stacks-project}).
	\end{exm}
	
	\begin{dfn}
		A triangulated category $\catT$ is \emph{compactly generated (as a triangulated category with small sums)} by a set $S$ of objects if all objects in $S$ are compact and if $\catT$ coincides with its smallest triangulated subcategory with small sums containing $S$.
	\end{dfn}
	
	\begin{prop}\label{genRigDA}
		The category $\RigDA_{\et}^{\dagger\eff}(S ,\Lambda)$ 
		is compactly generated (as a triangulated category with small sums) by motives $\Lambda(X )$ associated with  affinoid dagger varieties $X $    which are \'etale over some dagger poly-disc  $\B^{m\dagger}\times V $ for some open affinoid subspace $V\subset S$.
	\end{prop}
	
	\begin{proof}
		Since any smooth dagger variety is locally \'etale over a dagger poly-disc over a rational open of $S $ by Proposition \ref{smooth},  the set of functors $H_i\RHom(\Lambda(X ),\cdot)$ detect quasi-isomorphisms between \'etale local objects, by letting $X $ vary among  spaces of the prescribed form  and $i$ vary in $\Z$.  We are left to prove that the motive $\Lambda(X )$  of any affinoid dagger smooth variety $X $ is compact. Since $\Lambda(X )$ is compact in $\catD(\Psh(\RigSm^{\dagger}\!/S))$  and $\Sing^{\B^{1\dagger}}$ commutes with direct sums, it suffices to prove that if $\{\mcF_i\}_{i\in I}$ is a family of $\et$-local complexes, then also $\bigoplus_i\mcF_i$ is $\et$-local. 
		If $I$ is finite, the claim follows from the isomorphisms $H_{-n}\RHom(X ,\bigoplus_i\mcF_i)\cong\bigoplus_i \HH^{n}(X ,\mcF_i)\cong\HH^n(X ,\bigoplus_i\mcF_i)$. A coproduct over an arbitrary family is a filtered colimit of finite coproducts, hence the claim follows from \cite[Proposition 4.5.62]{ayoub-th2}.
	\end{proof}
	
	The previous proof 
	can be generalized to the following result.
	
	\begin{prop}
		The category $\RigDA_{\et}^{\eff}(\hat{S},\Lambda)$ 
		is compactly generated (as a triangulated category with small sums) by motives $\Lambda(X)$ associated with  affinoid dagger varieties $X$    which are \'etale over some dagger poly-disc  $\B^{m}\times V$ for some open affinoid subspace $V\subset S$. 
	\end{prop}
	
	\begin{cor}\label{genRigDArig}
		The category $\RigDA_{\et}^{\eff}(\hat{S},\Lambda)$ 
		is compactly generated by motives $\Lambda(\hat{X})$   associated with  affinoid varieties which admit a dagger structure.
	\end{cor}

	\begin{proof}
		This follows from the previous proposition and  Corollary \ref{etaleoverball}.
	\end{proof}

	The categories of motives have been introduced by imposing \'etale descent and homotopy invariance on  (co-)homological theories. If one wants to impose the extra condition that the Tate twist is invertible, they are not yet  enough. Nonetheless, there is a canonical way to do so (by means of the language of model categories) via the introduction of  \emph{spectra}. We refer to   \cite{hovey-sp} and \cite[Section 4.3]{ayoub-th2} for the details of this construction, that we simply apply to our dagger context.

	\begin{dfn}
		Consider the cokernel $L^\dagger$ in $\Psh(\RigSm^\dagger\!/S )$ of the map $\Lambda(S )\ra\Lambda(\mathbb{P}^{1\dagger}\times S )$ induced by the inclusion of the point $\infty$ in $\mathbb{P}^{1\dagger}$ and let $T^\dagger$ its shift $L^\dagger[-2]$ in $\Ch\Psh(\RigSm^\dagger\!/S )$.  It is a direct factor of a cofibrant object, hence cofibrant. We consider the category of (non-symmetric) spectra $\SSpect_{T^\dagger}\Ch_{\et,\B^{1\dagger}}\Psh(\RigSm^\dagger\!/S )$ and we denote by $\RigDA_{\et}^{\dagger}(S ,\Lambda)$ its homotopy category. 
	\end{dfn}
	
	\begin{rmk}
		The category $\RigDA_{\et}^{\dagger}(S ,\Lambda)$ is canonically equivalent to the homotopy category of symmetric spectra $\SSpect_{T^\dagger}^\Phi\Ch_{\et,\B^{1\dagger}}\Psh(\RigSm^\dagger\!/S )$ (see \cite[Proposition 4.3.47]{ayoub-th2}) and in particular inherits a monoidal structure, compatible with the canonical adjunction (see \cite[Lemma 4.3.24]{ayoub-th2})
		$$
		\adj{\LL\Sus_0}{\RigDA_{\et}^{\dagger\eff}(S ,\Lambda)}{\RigDA_{\et}^{\dagger}(S ,\Lambda)}{\RR\Ev_0}
		$$
	\end{rmk}

	The aim of the following part is to compare the categories of motives $\RigDA^{\dagger}$ and $\RigDA$ introduced above. We start by defining the canonical adjunction pair between them. The completion functor $\Spa R\mapsto\Spa\hat{R}$ defined on affinoid dagger spaces can be extended (by glueing) to a functor $l$ from dagger spaces to rigid analytic spaces (see \cite[Theorem 2.19]{gk-over}).
	
	\begin{prop}\label{Dff2}
		The completion functor $l\colon\RigSm^\dagger\!/S\rightarrow{\RigSm}\!/\hat{S}$ induces a Quillen adjunction
		\[
		\adj{l^*}{\Ch_{{\et},{\B}^{1\dagger }}\Psh(\RigSm^\dagger\!/S )}{\Ch_{{\et},{\B}^1}\Psh({\RigSm\!/\hat{S}})}{l_*}
		\]
		Moreover, the functor $\LL l^*$ is monoidal and  the functor $\RR l_*$ coincides with $l_*$.
	\end{prop}
	
	\begin{proof}
		The existence of the Quillen adjunction follows formally from the continuity of $l$ and the formula $l(\B^{1\dagger})\cong\B^1$. Also, the fact that $\LL l^*$ is monoidal follows from the formula $l^*(X \times_{S }Y )\cong l^*(X )\times_{\hat{S}}l^*(Y )$. We are left to prove that $\RR l_*\cong l_*$.
		
		By its very definition, $l_*$ preserves quasi-isomorphisms of complexes of presheaves. 
		We  claim that $l_*$ preserves also $\et$-weak equivalences. Fix any affinoid dagger space $X$ and any presheaf $\mcF$ on $\AffSm$. Let $R(X )$ resp. $R(\hat{X})$ be the class of covering families of $X $ resp. $\hat{X}$. By Corollary \ref{dagcov} we know that any covering of $\hat{X} $ can be refined into one  coming from a covering of $X $. This proves that the canonical map from $(l_*\mcF)^+(X)\colonequals \varinjlim_{\mcU \in R(X )}H^0(\mcU ,l_*\mcF)$ to $\mcF^+(\hat{X})\colonequals \varinjlim_{\mcU\in R(\hat{X})}H^0(\mcU,\mcF)$ is invertible. By the explicit construction of the sheafification functor as $\mcF\mapsto\mcF^{++}$ (see e.g. \cite[Tag 00W1]{stacks-project}) we then deduce that  $l_*$ commutes with the \'etale sheafification functor, and hence it also preserves  $\et$-weak equivalences as claimed.
		
		We also  remark that for any complex $\mcF$ and any dagger variety $X $ we obtain 
		$$
		\begin{aligned}
		(l_*\Sing^{\B^1}\mcF)(X )&=\Tot C_\bullet{\Hom}(\Lambda(\hat{X}\times\square^\bullet),\mcF)=\Tot C_\bullet{\Hom}(l_*\Lambda(X \times\square^{\dagger\bullet}),\mcF)\\
		&=\Tot C_\bullet{\Hom}(\Lambda(X \times\square^{\dagger\bullet}),l_*\mcF)=(\Sing^{\B^{1\dagger}}l_*\mcF)(X )
		\end{aligned}
		$$
		so that $l_*\Sing^{\B^1}\cong\Sing^{\B^{1\dagger}}l_*$. We then deduce by Proposition \ref{sing} that $l_*$ maps $\B^1$-weak equivalences to $\B^{1\dagger}$-weak equivalences, as wanted.
	\end{proof}

	\begin{cor}\label{Dff3}
		The natural functor $l\colon\RigSm^\dagger\rightarrow{\RigSm}$ induces a Quillen adjunction
		\[
		\adj{l^*}{\SSpect_{T^\dagger}\Ch_{{\et},{\B}^{1\dagger }}\Psh(\RigSm^\dagger\!/S )}{\SSpect_T\Ch_{{\et},{\B}^1}\Psh({\RigSm\!/\hat{S}})}{l_*}
		\]
		Moreover, the functor $\LL l^*$ is monoidal and  the functor $\RR l_*$ coincides with $l_*$.
	\end{cor}
	
	\begin{proof}
		Since $\LL l^*$ is monoidal and $ l^*T^\dagger\cong T$ the result follows formally from \cite[Proposition 5.3]{hovey-sp} and Proposition \ref{Dff2}.
	\end{proof}
	
	We will see that the following technical proposition implies the the fully faithfulness of the functor $\LL l^*$. It relies heavily on the approximation result (Proposition \ref{sollevacocu}) obtained in the previous section.
	
	\begin{prop}\label{qiso}
		Suppose that $S $ is affinoid. Let $X $  be an  affinoid dagger algebra smooth over $S $ with limit $\hat{X}$ and   $Y $ 
		be a dagger algebra, \'etale over a poly-disc  $\B^{m\dagger}$.  The canonical map 
		$$ (\Sing^{\B^{1\dagger}}\Lambda(Y ))(X )\ra (\Sing^{\B^{1}}\Lambda(\hat{Y}))(\hat{X})$$
		is a quasi-isomorphism.
	\end{prop}

	\begin{proof}
		We need to prove that the natural map
		\[
		\phi\colon N_\bullet\Lambda\Hom(X \times\square^\dagger,Y )\ra N_\bullet\Lambda\Hom(\hat{X}\times\square,\hat{Y})
		\]
		defines bijections on homology groups. 
		
		We start by proving  surjectivity. Suppose that $\beta\in\Lambda\Hom(\hat{X}\times\square^n,\hat{Y})$ defines a cycle in $N_n$ that is,  $\beta\circ d_{r,\epsilon}=0$ for $1\leq r\leq n$ and $\epsilon\in\{0,1\}$. This means that $\beta=\sum \lambda_k f_k$ with $\lambda_k\in\Lambda$, $f_k\in\Hom(\hat{X}\times\square^n,\hat{Y})$ and $\sum \lambda_k f_k\circ d_{r,\epsilon}=0$. This amounts to say that for every $k,r,\epsilon$ the sum $\sum\lambda_{k'}$ over the indices $k'$ such that $f_{k'}\circ d_{r,\epsilon}=f_{k}\circ d_{r,\epsilon}$ is zero. By Proposition \ref{sollevacocu}, we can find   maps $H_k\in\Hom(\hat{X}\times\square^n\times\B^1,\hat{Y})$ such that $i_0^*H=f_k$, $i_1^*H=\phi(\tilde{f}_{k})$ with $\tilde{f}_k\in\Hom(X \times\square^{n\dagger},Y )$ and $H_k\circ d_{r,\epsilon}=H_{k'}\circ d_{r,\epsilon}$ whenever ${f}_k\circ d_{r,\epsilon}={f}_{k'}\circ d_{r,\epsilon}$. 
		We denote by $H$ the cycle $\sum\lambda_k H_k\in\Lambda\Hom(\hat{X}\times\square^n\times\B^1,\hat{Y})$. 
		By \cite[Lemma 3.14]{vezz-fw}  we conclude that $i_1^*H$ and $i_0^*H$ define the same homology class, and therefore  $\beta$ defines the same class as $i_1^*H$ which is the image of a class in $\Lambda\Hom(X \times\square^{n\dagger},Y )$ as wanted. 
		
		We now turn to the injectivity.  Consider an element $\alpha\in\Lambda\Hom(X \times\square^{n\dagger},Y )$ such that $\alpha\circ d_{r,\epsilon}=0$ for all $r,\epsilon$ and suppose there exists an element $\beta=\sum\lambda_if_i\in\Lambda\Hom(\hat{X}\times\square^{n+1},\hat{Y})$ such that $\beta\circ d_{r,0}=0$ for $1\leq r\leq n+1$, $\beta\circ d_{r,1}=0$ for $2\leq r\leq n+1$ and $\beta\circ d_{1,1}=\phi(\alpha)$. Again, by Proposition \ref{sollevacocu}, we can find  maps $H_k\in\Hom(\hat{X}\times\square^{n+1}\times\B^1,\hat{Y})$ such that $H\colonequals\sum\lambda_k H_k$ satisfies $i_1^*H=\phi(\gamma)$ for some $\gamma\in\Lambda\Hom(X \times\square^{(n+1)\dagger},Y )$, $H\circ d_{r,0}=0$ for $1\leq r\leq n+1$, $H\circ d_{r,1}=0$ for $2\leq r\leq n+1$ and $H\circ d_{1,1}$ is constant on $\B^1$ and coincides with $\phi(\alpha)$. We conclude that $\gamma\in N_n$ and $d\gamma=\alpha$. In particular, $\alpha=0$ in the homology group, as wanted.
	\end{proof}

	We are now ready to prove the main result of this section.
	
	\begin{thm}\label{main}
		The  functors $ (\LL l^*,\RR l_*)$ define  triangulated, monoidal %
		equivalences
		$$
		\begin{aligned}
		\RigDA^{\dagger\eff}_{\et}(S ,\Lambda)&\stackrel{\sim}{\ra}{\RigDA^{\eff}_{\et}(\hat{S},\Lambda)}\\
		\RigDA^{\dagger}_{\et}(S ,\Lambda)&\stackrel{\sim}{\ra}{\RigDA_{\et}(\hat{S},\Lambda)}
		\end{aligned}
		$$
	\end{thm}
	
	\begin{proof}
		We start by proving the result on the effective categories.		We already showed in Remark \ref{genRigDArig} %
		that  the image by $\LL  l^*$ of the set of compact generators $\Lambda(Y )[i]$ as $Y $ varies in $\RigSm\Aff^\dagger\!/S$ and $i\in\Z$ (see Proposition \ref{genRigDA}) is a set of compact generators of $\RigDA^{\eff}_{\et}(\hat{S},\Lambda)$. By \cite[Lemma 1.3.32]{ayoub-rig} it then suffices to prove that the natural transformation $\id\Rightarrow\RR l_*\LL l^*$ is invertible. %

		By Proposition \ref{Dff2} the functor $\RR l_*$ is $l_*$. 
		Fix now a  dagger space $Y $ smooth over $S $. 
		We claim  that the motive $    l_*\LL  l^*\Lambda(Y )$ is isomorphic to $\Lambda(Y )$. By Proposition we can find a cover  $\{j_i\colon W_i \ra S \}$ by open affinoid subspaces of $S $ such that each $Y_i \colonequals Y \times_{S }W $ is \'etale over some dagger poly-disc $\B^{m\dagger}_{\hat{W}_i}$. By the dagger, effective version of \cite[Lemma 3.4]{ayoub-etale} 
		we obtain that the collection of functors $\{\LL j_i^*\}_i$ is conservative. In particular, we can prove that $\LL j_i^*\Lambda(Y )\cong \LL j_i^* l_*\LL  l^*\Lambda(Y )$. The functor $\LL j_i^*$ obviously commutes with $\LL l^*$. We claim that it also commutes with $l_*$. 
		We can  alternatively prove that $j_{i\sharp}$ commutes with $l^*$ (see Proposition \ref{6f}) and this is also straightforward. We are then left to prove that the canonical map 
		$$\LL j_i^*\Lambda(Y )\cong\Lambda(Y_i )\ra l_*\Lambda(\hat{Y}_i)\cong  l_*\LL  l^*\Lambda(Y _i)\cong \LL j_i^*l_*\LL  l^*\Lambda (Y )$$
		is invertible. 
		Therefore, we assume from now on that  the base dagger variety $S $ is affinoid, and that $Y $ is \'etale over a poly-disc. 	We will equivalently show that  $  l_*\Sing^{\B^{1}}\Lambda(\hat{Y})\cong \Sing^{\B^{1\dagger}}(Y )$. It is enough to show that for any affinoid dagger space   $X $ smooth over $S $ and with limit  $\hat{X}$, the canonical map
		$$(\Sing^{\B^{1\dagger}}\Lambda(Y ))(X )\ra (\Sing^{\B^{1}}\Lambda(\hat{Y}))(\hat{X})$$
		is a quasi-isomorphism, and this is true by means of Proposition \ref{qiso}.
		
		We finally conclude that the triangulated subcategory of $\RigDA^{\dagger\eff}_{\et}(S,\Lambda )$ of those objects $\mcF$ such that the canonical map  $\mcF\ra  l_*\LL  l^*\mcF$ is invertible contains $\Lambda(Y )[i]$ for any $Y $ smooth over $S $. By Proposition \ref{genRigDA} we deduce that this subcategory coincides with the whole $\RigDA^{\dagger\eff}_{\et}(S ,\Lambda)$ and therefore $  l_*\LL  l^*\cong\id$ as wanted. 
		
		We now turn to the assertion for the stable categories. From what we proved above we conclude that the Quillen adjunction of Proposition \ref{Dff2} is a Quillen equivalences. By means of \cite[Theorem 5.5]{hovey-sp} the same is true also for the adjunction of Corollary \ref{Dff3}, hence the claim. 
	\end{proof}

	\section{The Monsky-Washnitzer Realization Functor}
	
	It is possible to use the equivalence  obtained above to define some realization functors for algebraic motives, as well to give an explicit description of the motives representing   rigid cohomology or the overconvergent de Rham cohomology. In this respect, motives will be used as a convenient formalism allowing choices modulo homotopy, localizations and reductions to special cases. This language will provide extremely concise proofs of the functoriality and the finite-dimensionality for both the overconvergent de Rham cohomology and rigid cohomology.
	
	\begin{assu}
		From now on, we suppose that $K$ is a complete valued filed  of mixed characteristic $(0,p)$ with a valuation of rank $1$ (we do not suppose that the valuation is discrete). We recall that we denote by $K^\circ$ its ring of integers and by $k$ its residue field of characteristic $p$. 
	\end{assu}

	Fix a formal scheme $\mathcal{X}$ of topological finite type over $K^\circ$. We can adapt Definition \ref{motives} to the setting of algebraic varieties and define the category  $\DA^{\eff}_{\et}(\mathcal{X}_\sigma,\Lambda)$ of \'etale motives over $\mathcal{X}_\sigma$ (see e.g. \cite[Page 20]{ayoub-etale}). Analogously, we can define the category $\FormDA_{\et}^{\eff}(\mathcal{X},\Lambda)$  of \'etale motives of formal varieties over $\mathcal{X}$. If we let $\mathfrak{M}$ be the model category $\Ch(\Lambda\Mod)$, these constructions are denoted by $\SH^{\eff}_{\mathfrak{M}}(\mathcal{X}_\sigma)$ and $\FSH^{\eff}_\mathfrak{M}(\mathcal{X})$   in \cite[Definition 1.4.12]{ayoub-rig} respectively, with the only difference that we are considering the (finer) \'etale topology rather than the Nisnevich site. They are both monoidal categories with respect to the tensor product inherited by the direct product of varieties, and they also admit stable versions $\DA_{\et}(\mathcal{X}_\sigma,\Lambda)$ and $\FormDA_{\et}(\mathcal{X},\Lambda)$ obtained by inverting the Tate twists.  By  \cite[Corollaries 1.4.24 and 1.4.29]{ayoub-rig} the special fiber functor induces  (triangulated, monoidal)  equivalences:
	$$
	\begin{aligned}
	\LL(\cdot)_\sigma^*\colon \FormDA_{\et}^{\eff}(\mathcal{X},\Lambda)&\cong \DA_{\et}^{\eff}(\mathcal{X}_\sigma,\Lambda)\colon \RR(\cdot)_{\sigma*}\\
	\LL(\cdot)_\sigma^*\colon \FormDA_{\et}(\mathcal{X},\Lambda)&\cong \DA_{\et}(\mathcal{X}_\sigma,\Lambda)\colon \RR(\cdot)_{\sigma*}
	\end{aligned}
	$$
	On the other hand, the generic fiber functor induces  Quillen adjunctions:
	$$
	\begin{aligned}
	{\LL(\cdot)_\eta^*}\colon{\FormDA_{\et}^{\eff}(\mathcal{X},\Lambda)}\leftrightarrows{ \RigDA_{\et}^{\eff}(\mathcal{X}_\eta,\Lambda)}\colon{ \RR(\cdot)_{\eta*}}\\
	{\LL(\cdot)_\eta^*}\colon{\FormDA_{\et} (\mathcal{X},\Lambda)}\leftrightarrows{ \RigDA_{\et} (\mathcal{X}_\eta,\Lambda)}\colon{ \RR(\cdot)_{\eta*}}.
	\end{aligned}
	$$
	
	\begin{rmk}\label{comsums}
		The functor $(\cdot)_{\eta*}$ as well as the \'etale and $\B^1$-localizations commute with direct sums (see the proof of \cite[Proposition 4.18]{vezz-fw}). Therefore, the functor $\RR(\cdot)_{\eta*}$ also commutes with direct sums.
	\end{rmk}
	
	\begin{dfn}
		The \emph{effective Monsky-Washnitzer realization functor} is the monoidal triangulated functor obtained as the following composition {\small
		$$
		\MW^{\eff*}\colon\!\!\DA^{\eff}_{\et}(k,\Lambda)\xrightarrow[\sim]{\RR(\cdot)_{\sigma*}}\FormDA^{\eff}_{\et}(K^\circ,\Lambda)\xrightarrow{\LL(\cdot)_\eta^*}\RigDA^{\eff}_{\et}(K,\Lambda)\xrightarrow[\sim]{\RR l_*}\RigDA^{\dagger\eff}_{\et}(K,\Lambda).
		$$
	}
		It has a right adjoint $\MW_*^{\eff}$ induced by  $\RR(\cdot)_{\eta*}\colon \RigDA(K,\Lambda)^{\eff}_{\et}\ra\FormDA^{\eff}_{\et}(K^\circ,\Lambda)$. 
		
		This adjoint pair has also a stable version $(\MW^*,\MW_*)$ whose  left adjoint{\small
		$$
		\MW^{*}\colon\DA_{\et}(k,\Lambda)\xrightarrow[\sim]{\RR(\cdot)_{\sigma*}}\FormDA_{\et}(K^\circ,\Lambda)\xrightarrow{\LL(\cdot)_\eta^*}\RigDA_{\et}(K,\Lambda)\xrightarrow[\sim]{\RR l_*}\RigDA^\dagger_{\et}(K,\Lambda)
		$$}
		is called the \emph{stable Monsky-Washnitzer realization functor}.
	\end{dfn}
	
	\begin{rmk}
		Fix  a formal scheme $\mathcal{X}$ of topological finite type over $K^\circ$ with special fiber $\bar{\mcX}=\mcX_\sigma$ and such that its generic rigid fiber $\hat{X}=\mathcal{X}_\eta$ has a dagger structure $X $.  The previous definition can be generalized to the following functor (also admitting a right adjoint and a stable version):
		$$		
		\DA^{\eff}_{\et}(\bar{\mathcal{X}},\Lambda)\xrightarrow[\sim]{\RR(\cdot)_{\sigma*}}\FormDA^{\eff}_{\et}(\mathcal{X},\Lambda)\xrightarrow{\LL(\cdot)_\eta^*}\RigDA^{\eff}_{\et}(\hat{X},\Lambda)\xrightarrow[\sim]{\RR l_*}\RigDA^{\dagger\eff}_{\et}(X,\Lambda ).
		$$
	\end{rmk}
	
	We now remark that our construction overlaps with the classic definition of Monsky-Washnitzer.
	
	\begin{prop}\label{MWexp}
		Let $\mcX$ be a smooth formal scheme of topological finite type over $K^\circ$ with special fiber $\bar{\mcX}$ and generic fiber $\hat{X}$ admitting a dagger structure $X $. Then $\MW^{\eff*}\Lambda(\bar{\mcX})\cong\Lambda(X )$. %
	\end{prop}
	
	\begin{proof}
		This follows from the definition of $\MW^{\eff*}$ and the formulas $\LL(\cdot)_\sigma(\Lambda(\mcX))\cong\Lambda(\bar{\mcX})$, $\LL(\cdot)_\eta^*(\Lambda(\mcX))\cong\Lambda(\hat{X})$ and $\LL l^*(\Lambda(X ))\cong\Lambda(\hat{X})$.
	\end{proof}
	
	It is immediate to see that these functors can be used to define a  cohomology theory for varieties $\bar{X}$ over $k$ satisfying \'etale descent and homotopy invariance, and such that it coincides with the de Rham cohomology of the generic fiber $\mcX_\eta$ whenever $\bar{X}$ admits a smooth formal model $\mcX$ over $K^\circ$.  We now describe this construction and we will later show (Proposition \ref{rep}) that this definition overlaps with the classical definition of rigid cohomology.
	
	\begin{dfn}
		We denote by $\Omega^{1\dagger}$ the presheaf on smooth  dagger spaces $\Omega^1_{\mcO^\dagger\!/K}$ and by $\Omega^{q\dagger}$ its $q$-th exterior power. It associates to a smooth dagger space $X$ the space of overconvergent differential $d$-forms $\Omega^{d\dagger}_{X/K}(X)$. When restricted to the analytic site of a dagger space $X$ it defines a coherent $\mcO^\dagger$-module. For any fixed map $\Lambda\ra K$,  the \emph{ overconvergent de Rham complex } is the object  $\Omega^\dagger$ in $\Ch\Psh(\RigSm^\dagger\!/K)$ concentrated in negative degrees with $\Omega^{\dagger}_{-q}=\Omega^{q\dagger}$ for $q\geq0$ and with the usual differential maps. It associates to a smooth dagger variety $X$ the complex $0\ra\mcO^\dagger_{X/K}(X)\ra\Omega^{1\dagger}_{X/K}(X)\ra\Omega^{2\dagger}_{X/K}(X)\ra\ldots$ (see  \cite[Definition 7.5.11]{fvdp}). We also denote with $\Omega^\dagger$ the associated motive in $\RigDA^{\dagger\eff}_{\et}(K,\Lambda)$.  Similarly, we denote by $\Omega$  de Rham complex in $\Ch\Psh(\RigSm/K)$ [resp. $\Ch\Psh(\Sm/K)$] as well as the associated motive in $\RigDA^{\eff}_{\et}(K,\Lambda)$ [resp. in $\DA^{\eff}_{\et}(K,\Lambda)$].
	\end{dfn}

	\begin{lemma}\label{omegaet}
		For any $q$ the restriction of $\Omega^{q\dagger}$ to the small \'etale site of a smooth dagger space $X $ coincides with the sheaf induced by the coherent $\mcO^\dagger_{\hat{X}}$-module $\Omega^{q\dagger}_{X/K}$.
	\end{lemma}
	
	\begin{proof}
		Fix an \'etale morphism $f\colon Y \ra X $. We want to prove that $\Omega^{q\dagger}$ is isomorphic to the sheaf induced by $\Omega^{q\dagger}_{X }$. Up to considering a rational cover, we can also assume that $f$ is a composition of rational embeddings and finite \'etale maps between affinoid dagger spaces. We are then left to consider the case in which $f$ is finite \'etale. 
		
		By  Remark \ref{fet} the morphism $ \mcO^\dagger(X )\ra  \mcO^\dagger(Y )$ is finite \'etale, hence $\Omega^{q\dagger}(Y )=\Omega^{q}(\mcO^\dagger(Y ))\cong\Omega^q(\mcO^\dagger(X ))\otimes_{\mcO^\dagger(X )}\mcO^\dagger(Y )\cong f^*\Omega^{q\dagger}(Y )$ as wanted.
	\end{proof}
	
	We recall that the analytification functor $X\mapsto X^{\an}$ from algebraic varieties over $K$ to rigid analytic varieties over $K$ induces the following adjunction:
	$$
	\adj{\LL\Rig^*}{\DA^{\eff}_{\et}(K,\Lambda)}{\RigDA^{\eff}_{\et}(K,\Lambda)}{\RR\Rig_*}.
	$$
	which also admits a stable version (see \cite[Page 54]{ayoub-rig}).
	
	We will use the term ``representing a cohomological theory" in the following sense.
	
	\begin{dfn}Let $H^*$ be a functor on smooth varieties to graded $\Lambda$-modules (e.g. a cohomology theory). 
		We say that an object $M$ in an effective category of motives \emph{represents} %
		$H^*$ on smooth varieties if there is a canonical isomorphism $\Hom(\Lambda(X)[-i],M)\cong H^i(X)$ for all smooth varieties $X$ and integers $i$. 
	\end{dfn}
	
	\begin{rmk}\label{rmkrep}
		For any motive $N$ the map $N\otimes \Lambda(I)\ra N$ is a $I$-weak equivalence whenever $I$ is the contractible object we are considering. Also, for any \'etale hypercover $\mcU_\bullet\ra X$ the map $\Lambda(\mcU_\bullet)\ra\Lambda(X)$ is an \'etale-weak equivalence (for the notations, see e.g. \cite[Proposition 3.10]{vezz-fw}). We deuce  that $\Hom(\Lambda(\mcU_\bullet)[i],M)\cong\Hom(\Lambda(X)[i],M)\cong\Hom(\Lambda(X\times I)[i],M)$ and therefore any cohomology theory represented by a motive has \'etale descent and is homotopy-invariant with respect to $I$. Vice-versa, if $H^*$ has \'etale descent, then it is representable by a motive $M$ if and only if $\Hom(\Lambda(X)[-i],M)\cong H^i(X)$ holds for all $i$ and all  $X$ varying in a fixed full subcategory of smooth varieties which  generates the \'etale topos (see Proposition \ref{loctop}).
	\end{rmk}
	
	\begin{rmk}
		As soon as the functor $H^*$ as an extra structure or satisfies extra axioms, the previous proposition can be reinforced. For example, one may impose an isomorphism between the two $\delta$-functor structures (if $H^*$ has one).
	\end{rmk}
	
	\begin{rmk}
		Any cohomology theory on smooth algebraic varieties represented by $M$ can be canonically generalized to arbitrary (not necessarily smooth) quasi-projective varieties $\Pi\colon Y\ra \Spec K$ by putting $H^i(Y)=\Hom(\Lambda(Y)[-i],M)$ where $\Lambda(Y)\colonequals\Pi_!\Pi^!\Lambda(K)$ is defined using the extraordinary direct and inverse image functors $(\Pi_!,\Pi^!)$ which are part of the six-operation formalism of algebraic motives (see \cite{ayoub-th1} and \cite{ayoub-th2}). 
	\end{rmk}
	
	We compare the cohomology theories which naturally arise by considering $\Omega^\dagger$ and the functors $\MW^*$, $\LL\Rig^*$ and $\RR l_*$ with some classical ones: the  overconvergent de Rham cohomology, which we will denote by $H_{\odR}^*$, algebraic de Rham cohomology, which we will denote by $H_{\dR}^*$,   and rigid cohomology, which we will denote by $H_{\rig}^*$. For their definitions, we refer to the survey in \cite[Chapter 7]{fvdp}.
	
	\begin{prop}\label{rep}
		Suppose $\Lambda\subset K$.
		\begin{enumerate}
			\item The complex $\Omega^\dagger$ is $(\et,\B^{1\dagger})$-local and represents  overconvergent rigid cohomology $H^*_{\odR}$ on smooth dagger  varieties over $K$. 
			\item The motive $\LL l^*\Omega^\dagger$ represents  overconvergent rigid cohomology $H^*_{\odR}$ on smooth rigid analytic varieties over $K$. 
			\item The motive $\MW^{\eff}_*\Omega^\dagger $ represents rigid cohomology $H^*_{\rig}$ on smooth varieties over $k$.
			\item The motive $\RR\Rig_*\LL l^*\Omega^\dagger$ represents the de Rham cohomology $H^*_{\dR}$ on smooth varieties over $K$. %
		\end{enumerate}
	\end{prop}
	
	\begin{proof}
		Fix an affinoid smooth dagger space $X $ with limit $\hat{X}$. Each  $\Omega^{q\dagger}$ is a sheaf of coherent $\mcO^\dagger$-modules, hence it is acyclic for the \'etale topology by Proposition \ref{cohocoh} and Lemma \ref{omegaet}.
		From a spectral sequence argument, we conclude that $H_i(\Gamma(X ,\Omega^{\dagger}))=\HH^i_{\et}(X ,\Omega^\dagger)$ which shows that $\Omega^\dagger$ is $\et$-local and that $H_i\Gamma(X ,\Omega^\dagger)=H^i_{\odR}(X )$. %
		
		We now prove that $\Omega^\dagger$  is $\B^{1\dagger}$-local. From the computations above, this amounts to prove that $H^i_{\odR}(X \times\B^{1\dagger})\cong H^i_{\odR}(X )$ which is classical (see \cite[Theorem 5.4]{mw-fc1} or \cite[Theorem 4.12]{gk-over} and \cite[Example 1.8]{gk-dR}) in case $K$ has a discrete valuation, and follows from \cite[Corollary 5.5.2]{berk-int} in the general case, since any smooth affinoid space has a model over a discrete-valuation complete sub-field of $K$.
		
		We observe that for any smooth rigid affinoid space $\hat{X}$ with a chosen associated dagger structure $X $ by Theorem \ref{main} we see that:
		$$
		\begin{aligned}
		\RigDA^{\eff}_{\et}(K,\Lambda)(\Lambda(\hat{X})[-i],\LL l^*\Omega^\dagger)&\cong\RigDA^{\dagger\eff}_{\et}(K,\Lambda)(\RR j_*\Lambda(\hat{X})[-i],\RR l_*\LL l^*\Omega^\dagger)\\
		&\cong\RigDA^{\dagger\eff}_{\et}(K,\Lambda)(\Lambda(X )[-i],\Omega^\dagger)\cong H^i_{\odR}(\hat{X}) 
		\end{aligned} 
		$$ 
		with the composite isomorphism canonical on $\hat{X}$. This is enough to show that $\LL l^*\Omega^\dagger$ represents the overconvergent de Rham cohomology on rigid varieties over $K$ (see Remark \ref{rmkrep}). 
		
		Fix now a smooth scheme $\bar{X}$ over $k$ having a smooth formal model $\mcX$. %
		We have  by Theorem \ref{main} 
		$$
		\begin{aligned}
		\DA^{\eff}_{\et}(k,\Lambda)(\Lambda(\bar{X})[-i],\MW^{\eff}_*\Omega^\dagger)&\cong 
		\RigDA^{\dagger\eff}_{\et}(K,\Lambda)(\LL(\cdot)_\eta^*\RR(\cdot)_{\sigma*}\Lambda(\bar{X})[-i],\LL l^*\Omega^\dagger)\\
		&\cong H^i_{\odR}(\mcX_\eta)\cong H^i_{\rig}(\bar{X})
		\end{aligned}
		$$
		and the composite isomorphism is canonical on $\bar{X}$. Since rigid cohomology satisfies \'etale descent \cite{chiar-tsu}, %
		this implies that it is represented by  $\MW^{\eff}_*\Omega^\dagger $ (see Remark \ref{rmkrep}). 
		
		Let now $X$ be a smooth algebraic variety over $K$. 
		By \cite[Theorem 2.3]{gk-dR} we obtain a canonical sequence of isomorphisms
		$$
		\begin{aligned}
		\DA^{\eff}_{\et}(K,\Lambda)(\Lambda(X)[-i],\RR\Rig_*\LL l^*\Omega^\dagger)&\cong\RigDA^{\eff}_{\et}(K,\Lambda)(\Lambda(X^{\an})[-i],\LL l^*\Omega^\dagger)\\
		&\cong H^i_{\odR}(X^{\an})\cong H^i_{\dR}(X)
		\end{aligned}
		$$
		which proves the last statement. 
	\end{proof}

	\begin{rmk}
		The previous proof can easily be made independent on \cite[Corollary 5.5.2]{berk-int}. Indeed, the proof of \cite[Theorem 5.4]{mw-fc1} applied to $K^\circ$ and to $\pi=p$ does not use the fact that the $p$-adic norm is discrete.
	\end{rmk}
	
	\begin{rmk}
		As a whole, we have then provided new formulas computing rigid and  overconvergent de Rham cohomologies by using canonical functors on the categories of motives. The usual formulas, involving colimits over the possible choices of lifting and dagger structures (see \cite{besser}) are encoded in the description of the functors $\RR(\cdot)_{\sigma*}$ and $\RR l_*$ respectively.
	\end{rmk}
	
	\begin{rmk}
		Various authors have already given examples of complexes of presheaves representing rigid cohomology (see \cite{besser}, \cite{dm} and \cite{mr}). By means of the previous proposition,  one can show that such complexes are isomorphic to $\MW^{\eff}_*\Omega^\dagger $ in $\DA^{\eff}_{\et}(k,\Lambda)$.
	\end{rmk}
	
	By what proved above, the following definition is  well posed.
	
	\begin{dfn}\label{cohthys}
		Let $\Lambda$ be $K$ and $i$ be in $\Z$.
		\begin{enumerate}
			\item For any $M$  in $\RigDA_{\et}(K,\Lambda)$ we denote $$H^i_{\odR}(M)\colonequals\Hom(\RR l_*M[-i],\Omega^\dagger).$$
			\item For any $N$    in $\DA_{\et}(k,\Lambda)$ we denote $$H^i_{\rig}(N)\colonequals\Hom(\MW^{*}N[-i],\Omega^\dagger)\cong H^i_{\odR}((\LL(\cdot)_\eta^*\circ\RR(\cdot)_{\sigma*})(N)).$$
		\end{enumerate}
	\end{dfn}
	
	\begin{rmk}
		The above definition of rigid cohomology is obviously functorial. Whenever $k$ is a finite field, we then obtain automatically an action of Frobenius on each cohomology group $H^i_{\rig}(N)$ since  relative Frobenius maps are invertible in $\DA_{\et}(k,\Q)$ (this category is equivalent to the one \emph{with transfers}, see e.g. \cite[Section 16.2]{cd}).
	\end{rmk}

	We conclude by pointing out a straightforward consequence of our constructions, the comparison theorems of Gro\ss e-Kl\"onne \cite{gk-over}, \cite{gk-dR} and a theorem of Ayoub \cite[Theorem 2.5.35]{ayoub-rig}. It can be summarized by saying that the finite dimensionality of  the overconvergent de Rham, Monsky-Washnitzer and rigid cohomology can be formally deduced by the finite-dimensionality of the ``classic" de Rham cohomology for projective algebraic varieties in characteristic zero.

	\begin{dfn}
		Suppose $\Q\subset\Lambda$. We denote by $\DA_{\et,\gm}(K,\Lambda)$ [resp. $\RigDA_{\et,\gm}(K,\Lambda)$] the full triangulated subcategory of $\DA_{\et}(K,\Lambda)$ [resp. $\RigDA_{\et}(K,\Lambda)$] formed by compact objects (see Definition \ref{cpt}).
	\end{dfn}
	
	\begin{rmk}
		The category $\DA_{\et,\gm}(K,\Lambda)$ is denoted by $\DA_{\et}^{\ct}(K,\Lambda)$ in \cite{ayoub-th1} and \cite{ayoub-th2} where $\ct$ stands for "constructible" and $\gm$ for "geometric". The equivalence of all these notions follows from \cite[Proposition 2.1.24]{ayoub-th1} (see also \cite[Theorem 1.4.40]{ayoub-rig}).
	\end{rmk}
	
	\begin{exm}\label{excpt}
		For any quasi-projective variety of finite type $\Pi\colon X\ra\Spec k$  the motive $\Lambda(X)=\Pi_!\Pi^!\Lambda(K)$ is compact (see \cite[Scholium 2.2.34]{ayoub-th1}). For any smooth quasi-compact rigid analytic variety $X$ over $K$ the motive $\Lambda(X)$ is compact (see \cite[Proposition 1.2.34]{ayoub-rig}).
	\end{exm}
	
	We can now recall the theorem of Ayoub on a generating set for rigid analytic motives with rational coefficients.

	\begin{thm}[{\cite[Theorem 2.5.35]{ayoub-rig},\cite[Definition 14.1]{mvw}}]\label{ayoubgen1}
		Suppose  $\Q\subset\Lambda$. The category $\DA_{\et}^{\eff}(K,\Lambda)$ [resp. $\RigDA_{\et}^{\eff}(K,\Lambda)$] is compactly generated (as a triangulated category with small sums) by the motives $\Lambda(X)$ where $X$ runs among [analytifications of] smooth projective varieties over $K$. %
	\end{thm}

	\begin{rmk}\label{remayoubgen}
		The original version of the previous result considers the categories of motives \emph{with transfers} (defined in \cite{mvw} and \cite{ayoub-rig}). We can state it also for motives without transfers by means of the equivalences proved in\cite{vezz-DADM} (see also  \cite[Appendix B]{ayoub-etale}). The proof of the statement above is highly non-trivial, and uses the whole equipment of resolution of singularities, induction on dimension and localization.
	\end{rmk}
	
	\begin{cor} \label{ayoubgen}
		Suppose $\Q\subset\Lambda$. The category $\DA_{\et,\gm}(K,\Lambda)$ [resp. $\RigDA_{\et,\gm}(K,\Lambda)$] coincides with the  triangulated subcategory of $\DA_{\et }(K,\Lambda)$ [resp. $\RigDA_{\et}(K,\Lambda)$] closed under direct summands generated by the motives $\Lambda(X)(d)$ where $X$ runs among [analytifications of] smooth projective varieties over $K$ and $d$ runs in $\Z$. All its objects are strongly dualizable.
	\end{cor}
	
	\begin{proof}
		From the previous theorem, we deduce that the motives $\Lambda(X)(d)$ with $d\in\Z$ and with $X$ [analytification of] a smooth projective variety are generators of the stable category $\DA_{\et }(K,\Lambda)$ [resp. $\RigDA_{\et}(K,\Lambda)$]. The first statement follows then from   \cite[Proposition 2.1.24]{ayoub-th1} (sometimes referred to as the theorem of Neeman and Ravenel, see \cite{neeman-rav}). 
		Since $\LL\Rig^*$ is monoidal and the motive $\Lambda(X)(d)$ is strongly dualizable in $\DA_{\et,\gm}(K,\Lambda)$ for any smooth projective variety $X$ (see \cite[Lemma 1.3.29]{ayoub-rig}), we also deduce  the final claim.
	\end{proof}

	\begin{cor}\label{findim}
		Suppose $\Lambda=K$ and let  $M$ be in $\RigDA_{\et,\gm}(K,\Lambda)$. Then $H^i_{\odR}(M)$ is finite dimensional for all $i$ and  equal to $0$ for $|i|\gg 0$. In particular, if $N$ is in $\DA_{\et,\gm}(k,\Lambda)$ then $H^i_{\rig}(N)$ is finite dimensional for all $i$ and equal to $0$ for $|i|\gg 0$.
	\end{cor}
	
	\begin{proof}
		In order to prove the statement, it suffices to prove that the triangulated functor $M\mapsto\Hom_\bullet(M^\vee,\LL l^*\Omega^\dagger)$ from $\RigDA_{\et,\gm}(K,\Lambda)$ to $\catD(\Lambda)$ takes values in the triangulated subcategory of compact objects  in $\catD(\Lambda)$. By what showed above, it suffices to show that for any smooth projective variety $X$ the complex $\Hom_\bullet(\Lambda(X^{\an}),\LL l^*\Omega^\dagger)$ is compact. But this complex is, by Theorem \ref{main} and \cite[Theorem 2.26]{gk-over} canonically isomorphic to $\Omega(X^{\an})$  which is in turn quasi-isomorphic,  by \cite[Theorem 2.3]{gk-dR} to the   complex $\Omega(X)$ which is manifestly compact. 
		
		For the last assertion, by Definition \ref{cohthys} it suffices to remark that both the functor $\LL(\cdot)_\eta^*$ and the functor $\RR(\cdot)_{\sigma*}$ preserve compact objects: the former has a sum-preserving right adjoint (see Remark \ref{comsums}) and the latter is a triangulated equivalence.
	\end{proof}
	
	\begin{rmk}
		When we apply the previous corollary to the motives $M=\Lambda(X)$ for a quasi-projective variety $X$ over $k$ of finite type [resp. a smooth quasi-compact rigid analytic variety $X$ over $K$] (see Example \ref{excpt}) we obtain the finite dimensionality and the boundedness of the rigid [resp. overconvergent de Rham]   cohomology groups of $X$.
	\end{rmk}
	
	\begin{rmk}
		The proof consists in a big ``formal"  step, which reduces the assertion to the study  of $\Omega(X)$ for a smooth projective variety $X$ over $K$. We can therefore conclude that the finite dimensionality of the overconvergent and rigid cohomologies follows formally from the finite dimensionality of the cohomology of $\Omega(X)$ aka  the Betti cohomology of $X(\C)$ (fixing a map  of abstract fields $K\ra\C$ to an algebraically closed complete archimedean field $\C$). %
	\end{rmk}
	
	\begin{rmk}
		We  point out that we do not require that the valuation of $K$ is discrete, and that the previous proof easily generalizes to arbitrary cohomology theories represented in $\RigDA(K,\Q)$ which are finite dimensional on analytifications of projective smooth algebraic varieties over $K$.
	\end{rmk}
	
	Similarly, we obtain an alternative proof of a result of Gabber \cite[Corollary 5.5.2]{berk-int}. We now indicate with $\Omega_K^\dagger$ the complex $\Omega^\dagger$ to emphasize its dependence on the base field $K$.
	
	\begin{cor}
		Let $L/K$ be an extension of complete valued fields and suppose $\Lambda=L$. Let $f$ be the induced map $\Spa L\ra\Spa K$. The canonical map in $\RigDA^{\dagger\eff}_{\et}(K,\Lambda)$\[
		\Omega^\dagger_K\otimes_KL\ra \RR f_*\Omega^\dagger_{L}
		\]
		is an isomorphism.
	\end{cor}
	
	\begin{proof}
		We can prove that the cone of the map $C$ is zero, which by Theorem \ref{ayoubgen1} amounts to prove that $\Hom(\RR l_*\Lambda(X^{\an})[-i],C)=0$ for all smooth projective varieties $X/K$ and all $i$. This follows from the following isomorphisms {\small\[\Hom(\RR l_*\Lambda(X^{\an})[-i],\Omega^\dagger_K\otimes_KL)\cong H^{i}_{\dR}(X)\otimes_KL\cong H^i_{\dR}(X_{L})\cong\Hom(\RR l_*\Lambda(X^{\an})[-i],\RR f_*\Omega^\dagger_{L}).\]}
	\end{proof}
	
	In the same spirit, 
	inspired by the work of Cisinski-D\'eglise \cite{cd-mw}, we can also obtain the  K\"unneth formula, generalizing \cite[Section 9.4]{gk-fin} where the discrete-valuation case is considered.  
	
	\begin{cor}Suppose $\Lambda=K$. 
		For any two  motives $M, N$ in $\RigDA_{\et,\gm}(K,\Lambda)$ we have
		\[
		H^n_{\odR}(M\otimes N)\cong\bigoplus_{p+q=n}H^p_{\odR}(M)\otimes H^q_{\odR}(N)
		\]
	\end{cor}

	\begin{proof}
		It suffices to prove that the cohomological realization functor $M\mapsto\Hom_\bullet(M,\LL l^*\Omega^\dagger)$ is monoidal. The wedge product induces a multiplication of complexes of presheaves  $\Omega^{\dagger}\otimes\Omega^{\dagger}\ra\Omega^{\dagger}$ and therefore it induces a multiplication also on $\Omega^{\dagger}$ as a motive. We then obtain a composed natural transformation of bi-functors $$\Hom_\bullet(-,\LL l^*\Omega^\dagger)\otimes\Hom_\bullet(-,\LL l^*\Omega^\dagger)\Rightarrow\Hom_\bullet(-\otimes-,\LL l^*\Omega^\dagger\otimes\LL l^*\Omega^\dagger)\Rightarrow\Hom_\bullet(-\otimes-,\LL l^*\Omega^\dagger)$$ and we can prove it is invertible by checking this  on the set of generators of $\RigDA_{\et,\gm}(K,\Lambda)$ formed by analytifications of smooth  projective algebraic varieties. The result then follows from the usual K\"unneth formula for algebraic de Rham cohomology \cite{kunneth}.
	\end{proof}

	\appendix
	
	\section{Dagger Spaces and Inverse Limits of Adic Spaces}

	In this section, we study the connection between the language of dagger algebras  with the theory of adic spaces developed by Huber \cite{huber2}.  Our aim is to show that a dagger algebra is a presentation of the compactification of $\Spa(A,A^\circ)$ as an inverse limit of strict inclusions. We do not claim much originality here, as the results are straightforward  reformulations of  \cite{berkovich}, \cite{gk-over} and \cite{huber}.

	We start by recalling the notion of compactification of Huber. In order to state it, it is crucial to use the language of adic spaces, which allows more flexibility for the choice of the ring of integral functions associated with a Banach algebra. 
	
	In this Appendix, $K$ is any complete non-archimedean field with respect to a valuation of rank $1$ having a pseudo-uniformizer $\pi$ and ring of valuation $K^\circ$. 
	We recall that all our rigid analytic varieties over the base field $K$ (and hence morphisms between them) are separated and taut, and that we denote by $\Spa R$ the rigid analytic space $\Spa(R,R^\circ)$ whenever $R$ is a Banach $K$-algebra.

	\begin{dfn}[{\cite[Lemma 1.3.10]{huber}}]
		A map $f\colon X\ra Y$ of quasi-separated adic spaces over $K$ locally (topologically) of finite type (see \cite[Paragraph 1.1.13 and Definition 1.2.1]{huber}) is \emph{partially} \emph{proper} if for any valuation ring $L^+$ in a field $L$ and any commutative diagram
		$$
		\xymatrix{
			\Spa(L,L^\circ)\ar[r]\ar[d]	& X\ar[d]^f\\
			\Spa(L,L^+)\ar[r]^-{y}	& Y
		}
		$$
		there exists a  unique lift $\Spa(L,L^+)\ra X$ of $y$.
	\end{dfn}
	
	\begin{exm}
		The space $\B^1$ is not partially proper over $K$. Consider the total order in $\R_{>0}\times\delta^\Z$ induced by putting $1<\delta<\R_{>1}$ and let $||\cdot||$ be the Gauss norm on $K[\tau]$. Consider the valuation ring $K(\tau)^+$ in $K(\tau)$ induced by the following valuation, taking values in $\R_{>0}\times\delta^\Z\cup\{0\}$: 
		$$
		|\cdot|_{\infty}\colon f=\sum a_i \tau^i\mapsto ||f||\cdot\delta^{\max\{i\colon |a_i|=||f||\}}.
		$$
		The associated valuation of rank $1$ is the Gauss norm on $K(\tau)$ which defines a point in $\B^1$. Nonetheless, there is no lift of $\Spa(K(\tau),K(\tau)^+)\ra\Spa K$ to $\B^1$ since $|\tau|_{\infty}=\delta>1$.
	\end{exm}
	
	\begin{dfn}[{\cite[Theorem 5.1.5]{huber}}]
		The \emph{universal compactification} of a map   of rigid analytic varieties $f\colon X\ra S$ is a factorization $X\stackrel{j}{\ra} X^{\cp}_f\stackrel{f'}{\ra} S$ of adic spaces such that $j$ is locally closed, $f'$ is partially proper and such that for any other factorization $X\stackrel{h}{\ra} Y\stackrel{g}{\ra} S$ with $g$ partially proper, there exists a unique map $i\colon X^{\cp}_f\ra Y$ making the following diagram commute:
		$$
		\xymatrix{
			& X^{\cp}_f\ar[dr]^{f'}\ar[dd]^{i}& \\
			X\ar[ur]^j\ar[dr]^{h} && S\\
			& Y\ar[ur]^{g}&
		}
		$$
		The universal compactification of $X\ra\Spa K$ will be simply called the \emph{compactification} of $X$ and denoted by $X^{\cp}$.
	\end{dfn}
	
	\begin{exm}
		The universal compactification of a map of affinoid rigid analytic varieties $\Spa S\ra\Spa R$ induced by a map $\phi\colon R\ra S$ is given by the affinoid (yet not rigid analytic!) space $\Spa(S,S^+)$ where $S^+$ is the integral closure in $S$ of the ring $\phi(R^\circ)+S^{\circ\circ}$. In particular, the compactification of $\Spa R$ is the space $\Spa(R,R^+)$ where $R^+$ is the minimal choice among rings of integral elements in $R$ over $K$, namely the integral closure of $K^\circ+R^{\circ\circ}$ in $R$. It contains $\Spa R$ as an open dense subset.
	\end{exm}
	
	\begin{dfn}
		Let $X\subset Y$ be an open immersion of rigid analytic varieties over a variety $S$. We write $X\Subset_S Y$ if the inclusion factors over the adic compactification of $X$ over $S$ (see \cite[Theorem 5.1.5]{huber}). In case $S=\Spa K$ we simply write $X\Subset Y$. 
	\end{dfn}
	
	\begin{rmk}
		Let $Y\ra X $ be an open immersion of rigid varieties over $S$. Then $Y\Subset_S X$ if and only if the compactification of $Y$over $S$ coincides with the compactification of $Y$ over $X$. 
	\end{rmk}

	\begin{dfn}\label{interiors}
		Let $f\colon X\ra S$ be a morphism of rigid analytic varieties over a field $K$.
		The \emph{interior} $\Int(X/S)$ [resp. the \emph{border} $\del(X/S)$] of $f$ is [the complementary of]  the union of its open supsbaces $U$ such that $U\Subset_S X$. If $S=\Spa K$ we simply write $\Int(X)$ [resp.  $\del(X)$].
	\end{dfn}
	
	We recall that the Berkovich space $X^{\Berk}$ associated with a rigid analytic (taut) variety $X$ is the universal Hausdorff quotient of the topological space underlying $X$ (see \cite[Theorem 2.24]{scholze}). In particular, there is a continuous quotient morphism $\Berk\colon X\ra X^{\Berk}$. Our notations coincide with the one of \cite{ayoub-rig} by means of the following  interpretation in terms of Tate and Berkovich spaces.

	\begin{prop}\label{berkint}
		Let $X\subset Y$ be an open immersion of rigid analytic varieties over a variety $S$.  Then $X\Subset_SY$ if and only if  $X^{\Berk}$ lies in the Berkovich interior of $Y^{\Berk}$ over $S^{\Berk}$. 
	\end{prop}
	
	\begin{proof}
		By \cite[Proposition 2.5.17]{berkovich}	we can assume that all varieties are affinoid $X=\Spa(B,B^\circ)$, $Y=\Spa(A,A^\circ)$, $S=\Spa(C,C^\circ)$. By \cite[Proposition 2.5.2(d) and Proposition 2.5.9]{berkovich} $X$ lies in $\Int(Y/S)$ if and only if the image of $A^\circ$  in $B^\circ/B^{\circ\circ}$ is integral over the image of $C^\circ$. This amounts to say that $A^\circ$ is mapped in the  integral closure $B^+$ of  $C^\circ+B^{\circ\circ}$ in $B$ which is precisely the ring of integers of the compactification $\Spa(B,B^+)$ of $X$ over $S$.
	\end{proof}
	
	We then immediately obtain the following  result.
	\begin{cor}
		Let $f\colon X\ra S$ be a morphism of rigid analytic varieties over a field $K$. The space $\Int(X/S) $ is the inverse image via $X\ra X^{\Berk}$ of the Berkovich interior of $X^{\Berk}$ over $S^{\Berk}$.
	\end{cor}
	
	We also recall the following fundamental formula of Berkovich.
	
	\begin{dfn}
		A rigid analytic variety $X$ is \emph{good} if for any point $x\in X$ there exists an open subaffinoid of $X$ containing the closure of $\{x\}$ in $X$ (see \cite[Proposition 8.3.7]{huber}).
	\end{dfn}
	
	\begin{cor}[{\cite[Proposition 2.5.8(iii)]{berkovich}}]\label{berkform}
		Let $f\colon X\ra Y$ and $g\colon Y\ra S$ be two morphisms of {good} rigid analytic varieties over a field $K$. It holds
		$$\Int(X/S)=\Int(X/Y)\cap f^{-1}\Int(Y/S) .$$
	\end{cor}

	\begin{rmk}
		We recall that a morphism $X\ra Y$ of rigid analytic spaces is \emph{partially proper} if its border is empty, and it is \emph{proper} if it is partially proper and quasi-compact. Thanks to the properties above, these definitions coincide with Berkovich's and with Huber's. The results of \cite{temkin1} and \cite{temkin2} also show that the notion of properness  coincides with Kiehl's (see also \cite[Remark 1.3.19]{huber} for the discrete-valuation case).
	\end{rmk}

	\begin{prop}[{\cite[Proposition 2.1.16]{ayoub-rig}}]\label{cofinal}
		Let $X=U(f_i/g)$ be a rational subvariety of the affinoid space $X_1$. The sequence $X_h=U(\pi f_i^h/g^h) $ of rational subspaces of $X_1$ satisfies $X\Subset_{X_1}X_{h+1}\Subset_{X_1}X_h$ and is coinitial with respect to $\Subset_{X_1}$ among rational subspaces greater than $X$.  
	\end{prop}
	
	We remark that the category of adic spaces doesn't have all inverse limits. Nonetheless, there is a notion of \emph{being similar to the inverse limit} for an object $X$ having maps towards a directed system $\{X_{i+1}{\ra} X_{i}\}$ given by Huber (see \cite[Section 2.4]{huber}) and denoted by $X\sim\varprojlim X_h$. In this case, the \'etale topos of $X$ is the filtered limit topos of those associated with $X_i$ (\cite[Proposition 2.4.4]{huber}).
	
	\begin{prop}\label{limit}
		Let $X=\Spa \hat{R}$ be a rational subspace of $X_1=\Spa \hat{R}_1$ and let $X_h=\Spa \hat{R}_h$ be a   sequence of rational affinoids in $X_1$ totally ordered with respect to $\Subset_{X_1}$ and coinitial among the opens $W$ with $X\Subset_{X_1}W$. 
		Let $R$ [resp. $R^{+}$] be $\varinjlim \hat{R}_h$ [resp. $\varinjlim \hat{R}_h^\circ$]. If we endow $R$ with the topology induced by $\hat{R}$ 
		then $\Spa(R,R^{+})$ is an affinoid adic space with global sections equal to $\hat{R}$ and
		\[
		\Spa(R,R^{+})\sim
		\varprojlim_h\Spa \hat{R}_h.\] 
		Moreover if $\Spa(S,S^+)$ is an affinoid adic space over $K$ with $S^+ $ bounded is $S$ then 
		\[\Hom(\Spa(S,S^+),\Spa(R,R^{+}))\cong\varprojlim_h\Hom(\Spa(S,S^+),\Spa \hat{R}_h).\]
	\end{prop}
	
	\begin{proof}
		By Proposition \ref{cofinal} we can assume that $X=U(f_i/g)$ and $X_h=U(\pi(f_i^h/g^h)) $ as subspaces of $X_1$. 
		We first observe that $R$ is dense in $\hat{R}$. Indeed, elements in $\hat{R}_1\langle \upsilon\rangle/(g\upsilon-f)$ having a polynomial as representative are dense, so that in particular the image of $\hat{R}_1\langle \pi\upsilon\rangle$ which is included in $R$   is also dense.
		
		We now claim that the  ring $R^{+}$ is open and bounded in $R$. Since $\hat{R}$ is reduced, then $\hat{R}^\circ$ is bounded and the topology on $R$ is induced by the ring of definition $R\cap \hat{R}^\circ$. 
		We already proved in Proposition \ref{daggerballrep} the chain of inclusions
		\[
		\pi(R\cap \hat{R}^\circ)\subset R^{+}\subset R\cap \hat{R}^\circ
		\]
		and therefore our claim. 
		
		Since $R^{+}$ is open and integrally closed in $R$ the pair $(R,R^{+})$ is an affinoid pair. By what we proved above and \cite[Lemma 7.5.3]{fvdp} its completion coincides with the pair  $(\hat{R},\hat{R}^+)$ where $\hat{R}^+$ is the $\pi$-adic completion of $R^{+}$. As $\Spa \hat{R}$ is adic (that is,  the structure presheaf is a sheaf) then also  $\Spa(\hat{R},\hat{R}^+)=\Spa(R,R^{+})$ is adic (see \cite[Theorem 2.2]{huber2}). 
		
		Maps of affinoid spaces $\Spa(T,T^+)\ra\Spa(S,S^+)$ over $K$ for which $S^+$ and $T^+$ are bounded are uniquely determined by the maps of abstract $K^\circ$-algebras $S^+\ra T^+$. The last isomorphism follows then from our definitions and \cite[Proposition 2.1(i)]{huber2}. If we apply it to spectra of valued fields $\Spa(L,L^+)$ we deduce in particular $|\Spa(R,R^{+})|\cong\varprojlim_h|\Spa \hat{R}_h|$ and therefore $\Spa(R,R^{+})\sim\varprojlim_h\Spa \hat{R}_h$.
	\end{proof}

	We warn the reader that the completion of $R^{+}$ may vary among the rings of integral elements of $R$, as the next examples show. 
	
	\begin{exm}\label{badex}
		If we take $X=X_h=X_1$  
		we obtain $\Spa(R,R^{+})=\Spa(\hat{R},\hat{R}^\circ)$.
	\end{exm}
	
	\begin{exm}
		Suppose that $X_2\Subset X_1$. By Proposition \ref{berkform} we deduce $X_{h+1}\Subset X_h$ 
		and therefore $\mcO^\circ(X_h)\subset\overline{K^\circ+\mcO^{\circ\circ}(X_{h+1})}$ so that $R^+$ is contained in the algebraic closure of $K^\circ+\hat{R}^{\circ\circ}$ in $\hat{R}$. We conclude that $\Spa(R,R^+)$ is the compactification of $X$ over $K$.  
	\end{exm}
	
	We now make specific examples of this last situation.
	
	\begin{exm}\label{daggerball}
		Consider the rational inclusion $X=\B^n=\Spa K\langle\underline{\tau}\rangle\ra\Spa K\langle\pi\underline{\tau}\rangle=X_1\cong\B^n$ and let $X_h$ be the rational space $U(\pi^{1/h}\underline{\tau})\colonequals U\left((\pi\underline{\tau})^{h}/\pi^{h-1}\right) $ of $X_1$. The sequence $X_{h+1}\Subset_{X_1}X_h$ is coinitial among opens $W$ such that $X\Subset_{X_1}W$ and $X_{h+1}\Subset X_h$. 
		Moreover, $R$ coincides with $K\langle\underline{\tau}\rangle^\dagger$.
	\end{exm}
	
	We remark that Proposition \ref{limit} applied to the example \ref{daggerball} generalizes the claim at the end of \cite[Example 7.58]{wedhorn}. 
	This last example can be extended to the following situation. 
	
	\begin{exm}\label{daggerrat}
		Consider a rational inclusion $X=X_1(f_1,\ldots,f_m/g)\Subset X_1$ of affinoid rigid spaces.  By the proofs of \cite[Proposition 2.5.2, Proposition 2.5.9]{berkovich} we can suppose that there are presentations $$\mcO(X_1)=K\langle \rho_1^{-1}\tau_1,\ldots,\rho_n^{-n}\tau_n\rangle/I$$ 
		$$\mcO(X)=K\langle \pi^{-1}\rho_1^{-1}\tau_1,\ldots,\pi^{-1}\rho_n^{-1}\tau_n,\upsilon_1,\ldots,\upsilon_m\rangle/((\upsilon_i g-f_i)+I)$$ 
		with $\rho_i\in\sqrt{K^\times}$. 
		We then define $X_h$ to be the rational subset of $X_1$ with 
		$$\mcO(X_h)=K\langle \pi^{\frac{1}{h}-1}\rho_1^{-1}\tau_1,\ldots,\pi^{\frac{1}{h}-1}\rho_n^{-1}\tau_n,\pi^{\frac{1}{h}}\upsilon_1,\ldots,\pi^{\frac{1}{h}}\upsilon_m\rangle/((\upsilon_i g-f_i)+I).$$ 
		The sequence $X_{h+1}\Subset X_h$  is coinitial among opens $W$ such that $X\Subset_{X_1}W$. 
		We also obtain 
		$$R=K\langle (\pi\rho_1)^{-1}\tau_1,\ldots,(\pi\rho_n)^{-1}\tau_n,\upsilon_1,\ldots,\upsilon_m\rangle^\dagger\!/((\upsilon_i g-f_i)+I)$$ which is a dagger algebra. 
	\end{exm}

	We are then inclined to make the following definition.

	\begin{dfn}
		Fix an affinoid rigid space $X$. A  \emph{presentation of a  dagger structure on $X$} is a pro-affinoid variety $\varprojlim X_h$ where $X$ and all $X_h$ are rational subspaces of $X_1$, such that 
		$X\Subset X_{h+1}\Subset X_h$  
		and the system is coinitial among rational subsets of $X_1$ containing $X$ in their interior.    A \emph{morphism} of presentations between $\varprojlim X_h$ and $\varprojlim Y_k$ is a morphism of pro-objects, that is,  an element of $\varprojlim_k\varinjlim_h\Hom(X_h,Y_k)$.  
	\end{dfn}
	
	\begin{exm}
		The dagger poly-disc $\B^{n\dagger}$ has the presentation $\varprojlim X_h$ described in  Example \ref{daggerball}. 
	\end{exm}
	
	\begin{rmk}
		The system $X_h$ of \ref{badex} is not a presentation of a dagger structure on $X$ since $X_{h+1}$ is not contained in the interior of$X_h$.
	\end{rmk}

	We summarize the previous discussion in the following proposition, drawing the link between the definition of dagger algebras and the language of (weak) inverse limits of adic spaces due to Huber.
	
	\begin{prop}\label{moraffdag}
		Let $\hat{X}=\Spa(\hat{R},\hat{R}^\circ)$,  be an affinoid space and let $\varprojlim {X}_h$ be a presentation of a dagger structure on $\hat{X}$.
		\begin{enumerate}
			\item $\varinjlim \mcO( {X}_h)$ is a  dagger algebra $R$ dense in $\hat{R}$. 
			\item\label{moraffdag1} The functor $\varprojlim {X}_h\mapsto \Spa^\dagger R$ induces an equivalence of categories between  dagger affinoid spaces $\Aff^\dagger$ as introduced in Definition \ref{defdag} and their presentations.
			\item  $\hat{X}^{\cp}=\Spa(\hat{R},\hat{R}^+)\cong\Spa(R,R^{+})\sim\varprojlim \hat{X}_h$ where $\hat{R}^+$ is the integral closure of ${K^\circ+\hat{R}^{\circ\circ}}$ in $\hat{R}$ and $R^+$ is $\varinjlim\hat{R}_h^\circ$.
			\item If $\Spa(T,T^+)$ is an affinoid adic space with $T^+ $ bounded then 
			\[\Hom(\Spa(T,T^+),\Spa(R,R^{+}))\cong\varprojlim_h\Hom(\Spa(T,T^+),\Spa \hat{R}_h).\]
		\end{enumerate}
	\end{prop}
	
	\begin{proof}
		The first claim follows from Example \ref{daggerrat}. The fully faithfulness of the functor in the second claim follows from Remark \ref{daggerballrep}. Its essential surjectivity is immediate: fix a dagger algebra $R$ and an integer $H$ such that  $R=K\langle\underline{\tau}\rangle^\dagger/(a_i)$ with $a_i\in K\langle \pi^{1/H}\underline{\tau}\rangle$. Then $R$ is the image of the sequence $X_{h+1}\Subset X_h$ with $X_h=\Spa K\langle \pi^{1/(H+h)}\underline{\tau}\rangle /(a_i)$. The last two claims follow from Proposition \ref{limit}.
	\end{proof}

	\begin{rmk}
		Even though $\hat{X}^{\cp}$ is a (weak) inverse limit of the spaces ${X}_h$ as adic space, morphisms between two presentations of dagger spaces $X $ and $Y$   do not coincide in general  with morphisms between $\hat{X}^{\cp}$ and $\hat{Y}^{\cp}$ as the latter coincide with morphisms from $\hat{X}$ to $\hat{Y}$.
	\end{rmk}
	
	We can also promote the equivalence of categories between dagger spaces and their presentations to an equivalence of topoi, using the following definitions.

	\begin{dfn}
		Let $\mathbf{P}$ be a property of morphisms of rigid spaces.
		We say that a morphism of pro-rigid spaces $\phi \colon X \ra Y  $ has the property $\mathbf{P}$ if  $X \cong\varprojlim X_h$, $Y^\dagger\cong\varprojlim Y_h $ and $\phi=\varprojlim \phi_h$ with  $\phi_h\colon X_h\ra Y_h$ having the property $\mathbf{P}$.  
		We say that a collection of open morphisms of pro-rigid spaces $\{\phi_i\colon \varprojlim_h U_{ih}\ra X \}_{i\in I}$ is a \emph{ cover} if $X\Subset \bigcup_i\im(U_{ih})$ for all $h$.
	\end{dfn}

	\begin{rmk}
		In particular, we have defined open immersions, smooth and \'etale morphisms of presentations of affinoid dagger spaces. In that case, as the morphisms $\hat{X}\subset X_h$ are open immersions (hence \'etale) we deduce that if a morphism $X \ra Y $ is an open immersion [resp. smooth resp. \'etale] then the associated morphism $\hat{X}\ra\hat{Y}$ also is.
	\end{rmk}

	\begin{rmk}
		The topology induced by  covers of open immersions is  generated by  covers of rational embeddings.
	\end{rmk}
	
	\begin{prop}\label{dagcovpre}
		Any  family of \'etale maps of presentations of affinoid dagger spaces $\{\varprojlim_h U_{ih}\ra\varprojlim_h X_h\}_{i\in I}$  inducing a covering of $\hat{X}$   is  a covering.
	\end{prop}
	
	\begin{proof}
		It suffices to show that if $U\Subset U_1$ and $f\colon U_1\ra X_1$ is an \'etale map, then $f(U)\Subset f(U_1)$. By \cite[Proposition 2.5.17]{berkovich} we can consider only the case in which $f$ is an open immersion, which is clear, and the case in which $f$ is  finite \'etale, which we now examine.
		
		Since $U_1\ra f(U_1)$ is finite, we deduce from \cite[Proposition 2.5.8(iii) and Corollary 2.5.13(i)]{berkovich} that $\Int(U_1)=f^{-1}\Int(X_1)$. Therefore since $U\Subset U_1$ we deduce $f(U)\Subset f(U_1)$ as wanted.
	\end{proof}

	\begin{cor}\label{eqtop2}
		Let $X $ be an affinoid dagger variety with a presentation $\varprojlim X_h$ and limit $\hat{X}$. 
		The maps of the small rational and \'etale sites $\hat{X}\ra X \ra\varprojlim X_h$ induces  equivalences on the associated topoi.
	\end{cor}
	
	\begin{proof}
		It suffices to use the criteria of \cite[Appendix A]{huber}.
	\end{proof}
	
	\begin{rmk}
		The content of the previous proposition may seem to clash with the result of Huber \cite[Proposition 2.4.4]{huber} giving an equivalence between the \'etale topos of $\hat{X}^{\cp}$ and the direct limit topos $\varprojlim  \Sh_{\et}(X_h)$. The point is that the \'etale site $\varinjlim X_{h,\et}$ giving rise to the direct limit topos $\varprojlim \Sh_{\et}(X_h)$ is not equivalent to the \'etale site of the presentation. As a hint of this fact, consider for example the constant system $U_h=\hat{X}$ which is not a presentation of a dagger structure and hence does not  define an open of $\hat{X}  $. 
	\end{rmk}
	
	In order to be consistent with the ``pro-objects" approach, we also introduce dagger spaces with a functorial perspective. We recall that the rational topology on dagger affinoid spaces is sub-canonical, that is,  the presheaf  $\mcF_X$ over $\Aff^\dagger$ represented by an affinoid dagger variety $X$ is a sheaf (this follows from Proposition \ref{eqrational}). By abuse of notation, we sometimes denote it by $X$.
	
	\begin{dfn}
		A morphism $\mcF\ra\mcG$ of sheaves of sets over $\Aff^\dagger$ with the rational topology \emph{has the property $\mathbf{P}$ }if for any morphism $X \ra \mcG$ from a representable one, the pull-back is representable and the morphism $\mcF\times_{\mcG}X \ra X $ has the property $\mathbf{P} $. A collection of morphisms $\{\mcF_i\ra\mcG\}$ is a \emph{cover} if $\bigsqcup\mcF_i\ra\mcG$ is an epimorphism of sheaves. A [\emph{smooth}] \emph{functorial dagger space} is a sheaf $\mcF$ over $\Aff^\dagger$ with a cover of open immersions $\{U_i\ra\mcF\}$ where each $U_i$ is represented by a [smooth] affine dagger space.
	\end{dfn}

	\begin{rmk}
		The category of functorial dagger spaces has fiber products, generalizing fiber products of affinoid dagger spaces.
	\end{rmk}

	\begin{rmk}
		\'Etale and open covers define a topology on functorial dagger spaces. 
	\end{rmk}

	The functor $l\colon \Aff^\dagger\ra\Aff$ induced by $X \mapsto \hat{X}$ is continuous. If we embed the category of affinoid varieties in  the category of locally ringed spaces $\LRS$, we therefore obtain a left adjoint Kan extension functor:
	$$|\cdot|\colon\Sh(\Aff^\dagger)\rightleftarrows\LRS$$
	such that $|X |=(|\hat{X}|,\mcO_{\hat{X}})$ for any $X\in \Aff^\dagger$ (we denote by $X$ also the sheaf  represented by $X$).
	
	\begin{prop}
		Let $\mcF$ be a [smooth] functorial dagger space with an open cover by dagger affinoid spaces $\{U_i \ra\mcF\}$. 
		The ringed space $|\mcF|$ is a [smooth] rigid analytic space covered by $\hat{U}_i$ and endowed with an extra sheaf $\mcO^\dagger$ such that $\mcO^\dagger|_{\hat{U}_i}$ is the sheaf $\mcO^\dagger_{U_i }$ introduced in Definition \ref{Odag}.
	\end{prop}
	
	\begin{proof}
		By \cite[IV.7.3 and A.1.1]{mamo} we can write a coequalizing diagram of sheaves
		$$
		\bigsqcup (U_i \times_{\mcF}U_j )\rightrightarrows\bigsqcup U_i \ra\mcF\ra0
		$$
		whose arrows on the left are open immersions of representable sheaves. By applying the left adjoint functor $|\cdot|$ 
		we deduce that $|\mcF|$ is obtained by gluing the analytic spaces $\hat{U}_i$ over open immersions and therefore is an analytic space. The sheaf $\mcO^\dagger$ can be  defined using the formula in the statement.
	\end{proof}
	
	\begin{rmk}\label{covcov}
		In particular, we proved that a morphism of affine dagger spaces $U \ra X $ is open or rational if an only if the associated map of representable sheaves is, and that a collection $\{U_i \ra X \}$ of morphisms of affine dagger spaces is an open cover if and only if the induced collection of morphisms of sheaves is.
	\end{rmk}

	We conclude that the functor $|\cdot|$ factors over the category of dagger spaces defined by Gro\ss e-Kl\"onne \cite{gk-over}.
	
	\begin{cor}
		The category of functorial dagger spaces is equivalent to the category of dagger spaces of Definition \ref{dagsp}, and this equivalence preserves rational/open immersions and covers.
	\end{cor}
	
	\begin{proof}
		If $X$ is a dagger space with a dagger affinoid covering $\{U_i\ra X\}$ we can consider the sheaf $\mcF_X$ it represents on the category of affinoid dagger spaces. Since the functor $X\mapsto\mcF_X$ is a right adjoint, it preserves intersections. We then conclude by Yoneda Lemma and Remark \ref{covcov} that $\{\bigsqcup U_i\ra \mcF_X\}$ is an epimorphism of open immersions so that $\mcF_X$ is a functorial dagger space. It is immediate to see that $X\mapsto\mcF_X$ and $\mcF\mapsto|\mcF| $ are quasi-inverse functors. The second part of the statement follows from Remark \ref{covcov}.
	\end{proof}
	
		\section*{Acknowledgements}
	This paper was been written during my stay at the IRMAR of Rennes, funded by the Lebesgue Center of Mathematics. I am grateful to the hosting organizations for their support. 
	I  thank Jean-Yves Etesse for pointing out some relevant results in the literature. 
	I also  express my gratitude to Joseph Ayoub, Fr\'ed\'eric D\'eglise, Elmar Gro\ss e-Kl\"onne, Bernard Le Stum and an anonymous referee for their precious remarks on earlier versions of this work.

 \end{document}